\documentclass[sn-mathphys,Numbered]{sn-jnl}
\usepackage{graphicx}%
\usepackage{multirow}%
\usepackage{amsmath,amssymb,amsfonts}%
\usepackage{amsthm}%
\usepackage{mathrsfs}%
\usepackage[title]{appendix}%
\usepackage{xcolor}%
\usepackage{textcomp}%
\usepackage{manyfoot}%
\usepackage{booktabs}%
\usepackage{mathtools}
\usepackage{algorithm}%
\usepackage{algorithmicx}%
\usepackage{algpseudocode}%
\usepackage{listings}%

\theoremstyle{thmstyleone}%
\newtheorem{theorem}{Theorem}
\newtheorem{lemma}{Lemma}
\newtheorem{assumption}{Assumption}
\theoremstyle{thmstyletwo}%
\newtheorem{example}{Example}%
\newtheorem{remark}{Remark}%

\theoremstyle{thmstylethree}%
\newtheorem{definition}{Definition}%

\begin{document}

\title[Data-driven Approximation of Distributionally Robust Chance Constraints using Bayesian Credible Intervals]{Data-driven Approximation of Distributionally Robust Chance Constraints using Bayesian Credible Intervals}

\author*[]{\fnm{Zhiping} \sur{Chen$^*$}}\email{zchen@mail.xjtu.edu.cn}

\author[]{\fnm{Wentao} \sur{Ma}}\email{mwtmwt7@stu.xjtu.edu.cn}
\author[]{\fnm{Bingbing} \sur{Ji}}\email{bingji0225@stu.xjtu.edu.cn}

\affil[]{\orgdiv{School of Mathematics and Statistics}, \orgname{Xi’an Jiaotong University},  \city{Xi’an}, \postcode{710049},  \country{China}}

\affil[]{\orgdiv{Center for Optimization Technique and Quantitative Finance}, \orgname{Xi’an International Academy for Mathematics and Mathematical Technology}, \city{Xi’an}, \postcode{710049},  \country{China}}

\abstract{The non-convexity and intractability of distributionally robust chance constraints make them challenging to cope with. From a data-driven perspective, we propose formulating it as a robust optimization problem to ensure that the distributionally robust chance constraint is satisfied with high probability.
To incorporate available data and prior distribution knowledge, we construct ambiguity sets for the distributionally robust chance constraint using Bayesian credible intervals. We establish the congruent relationship between the ambiguity set in Bayesian distributionally robust chance constraints and the uncertainty set in a specific robust optimization. In contrast to most existent uncertainty set construction methods which are only applicable for particular settings, our approach provides a unified framework for constructing uncertainty sets under different marginal distribution assumptions, thus making it more flexible and widely applicable. Additionally, under the concavity assumption, our method provides strong finite sample probability guarantees for optimal solutions.
The practicality and effectiveness of our approach are illustrated with numerical experiments on portfolio management and queuing system problems. Overall, our approach offers a promising solution to distributionally robust chance constrained problems and has potential applications in other fields.}


\maketitle

\section{Introduction}\label{sec1}
	Robust optimization (RO) has become a main tool nowadays for addressing uncertainty in decision-making problems, e.g., network design \cite{ref14}, inventory management \cite{222}, and portfolio selection \cite{333}. It aims at determining an optimal decision by considering the worst-case realization among all possible realizations of the uncertainty parameters, called the uncertainty set.  More specifically, if we denote the constraint function of a decision-making problem with the uncertain parameter $\pmb{\xi}\in\mathbb{R}^d$ as  $g({\pmb{\xi}},\pmb{x})\leq0$, its RO counterpart can be formulated as
		\begin{equation}
		g(\pmb{\xi},\pmb{x})\leq0,\ \ \forall\pmb{\xi}\in\Xi. \label{1.1}
	\end{equation}
Here, $\pmb{x}\in\mathbb{R}^k$ is the decision variable, and the uncertainty set $\Xi\subseteq\mathbb{R}^d$ of ${\pmb{\xi}}$ is pre-specified by the decision-maker. For example, if  $\tilde{\pmb{\xi}}$\footnote{In this paper, we denote a generic random parameter as  $\tilde{\pmb{\xi}}$ and its possible realization as  ${\pmb{\xi}}$ for the sake of distinction.} is a random parameter, $\Xi$ can be chosen as the support set of $\tilde{\pmb{\xi}}$. The decision variable $\pmb{x}$ satisfying (\ref{1.1}) ensures that the constraint is not violated for any possible realization within the uncertainty set $\Xi$.

It is easy to see that selecting an appropriate uncertainty set is critical for the success of RO models. On the one hand, an uncertainty set that includes all possible sample realizations of the random vector $\tilde{\pmb{\xi}}$ guarantees that all uncertainties are taken into account in (\ref{1.1}). On the other hand, an extremely small appearing probability for a specific realization of $\tilde{\pmb{\xi}}$ may lead to overly conservative decisions. As highlighted in \cite{review2}, poorly chosen uncertainty set can make robust optimization models computationally intractable or overly conservative under (\ref{1.1}). Therefore, only with well-chosen uncertainty sets can the solutions of RO problems perform as well or better than other optimization approaches.

In order to cope with the above issues, an alternative framework can be adopted to reduce conservatism. This approach replaces (\ref{1.1}) with a chance constraint that allows a small probability of violation, with respect to the distribution $\mathbb{P}^c$ of $\tilde{\pmb{\xi}}$, of the inequality constraint. 
Nevertheless, the feasible set defined by chance constraints is very often non-convex, which leads to the computational difficulty for solving optimization problems with chance constraints directly. Thus, many methods, such as \cite{ref6,ref22}, utilize the idea of RO to construct an uncertainty set $\Xi_\epsilon$ such that a feasible solution $\pmb{x}^*$ to the constraint of the form (\ref{1.1}) with the uncertainty set $\Xi_\epsilon$ will also be feasible to $g(\tilde{\pmb{\xi}},\pmb{x}^*)\leq0$ with probability at least $1-\epsilon$ with respect to $\mathbb{P}^c$. In other words, we have the following relationship:
\begin{equation}
	\text{if }g({\pmb{\xi}},\pmb{x}^*)\leq0,\ \forall \pmb{\xi}\in\Xi_\epsilon,\ \text{then }\mathbb{P}^c(g(\tilde{\pmb{\xi}},\pmb{x}^*)\leq0)\geq1-\epsilon,\label{1.3}
\end{equation}
where the parameter $\epsilon$ controls the risk of violating the uncertain constraint $g(\tilde{\pmb{\xi}},\pmb{x})\leq0$. 
Unfortunately, in many applications, it is hardly possible to obtain complete knowledge about the true distribution $\mathbb{P}^c$. 
To deal with this issue in (\ref{1.3}), the distributionally robust optimization (DRO) approach has recently attracted increasing attention, which assumes that the possible distributions of $\tilde{\pmb{\xi}}$ belong to a so-called ambiguity set $\mathcal{P}$ instead of being precisely known. The combination of DRO with chance constraints gives rise to the distributionally robust chance constraint (DRCC) approach, which provides a conservative approximation to the chance constraint in (\ref{1.3}) and has been extensively studied in the literature (e.g., \cite{ref17,ref102,ref213}). With this paradigm, (\ref{1.3}) can be transformed as follows:
 \begin{equation}
	\text{if }g({\pmb{\xi}},\pmb{x}^*)\leq0,\ \forall \pmb{\xi}\in\Xi_\epsilon,\ \text{then }\inf_{\mathbb{P}\in\mathcal{P}}\mathbb{P}(g(\tilde{\pmb{\xi}},\pmb{x}^*))\leq0)\geq1-\epsilon,\label{1.4}
\end{equation}
which requires that the chance constraint must hold for all possible probability distributions belonging to the ambiguity set $\mathcal{P}$.

The construction of a reasonable ambiguity set $\mathcal{P}$ is undoubtedly very important for DRO. Several types of ambiguity sets have been proposed in the literature, including those based on moments, unimodality, or the Wasserstein distance (see e.g.,\cite{ Wass1,ref9_2,ref19}). Regardless of which type we choose, the goal is to create an ambiguity set that contains the true probability distribution with high probability and can provide some guarantees on the out-of-sample performance of the distributionally robust solution. The DRO problems in these studies are tractable and with strong asymptotic performance guarantees.  However, similar guarantees may not be available in the case of finite samples.
A good way to achieve the goal of finite sample guarantees is to sufficiently utilize the prior knowledge and the structural features of the underlying distribution. In this aspect, there are some studies that construct ambiguity sets from the perspective of hypothesis testing in statistics. 

Under the assumption that the random variables are Gaussian and mutually independent, the chi-square hypothesis test is used in \cite{ref8,ref19_2} to construct confidence intervals around the sample means and variances, and then to generate an ambiguity set. Some reformulations of (\ref{1.4}) have been derived for cases where $\tilde{\pmb{\xi}}$ follows a multivariate normal distribution without the independence assumption (e.g., \cite{ref17_2,ref44,refXuHuan}). Although such studies enable us to characterize the finite sample and asymptotic guarantees, they are limited to the family of normal distributions and may not be applicable to complex distributions. 
Meanwhile, Jiang and Guan (\cite{ref15}) studied the equivalent reformulation of DRCC problems under the confidence intervals with Kullback-Leibler (KL) divergence, but did not investigate finite sample guarantees or asymptotic convergence. The KL divergence leads to an ambiguity set of discrete probability distributions with finite support, which is not suitable for describing complex continuous distributions in the real world. The work in \cite{main} adopted different statistical hypothesis tests to convert data and chance constraints into DRO constraints.
It focuses only on the tractability and the finite-sample guarantee rather than asymptotic performance of the DRCC formulation (\ref{1.4}).  Moreover, the goodness-of-fit hypothesis test used in \cite{main} to construct uncertainty constraints may not ensure an ambiguity set of distributions which can contain the true probability distribution with high probability, particularly when dealing with complex distributions (e.g., multi-variate distributions with dependent marginal distributions) due to the lack of proper asymptotic theory.

As mentioned above, using fixed prior knowledge in statistical inference to construct ambiguity sets can improve the performance. However, incorporating adjustable prior knowledge into probability can be challenging. The existing literature, based on the classical statistical perspective of distributions, cannot overcome this difficulty. Bayesian statistics (\cite{bayes}), on the other hand, treats probability as a personal belief based on an individual's experience about the possibility of an event, known as a priori probability. By flexibly setting prior probability, which reflects one's understanding of a random phenomenon, we can obtain a good ambiguity set $\mathcal{P}$ for the possible distributions, thus leading to a well-performing uncertainty set $\Xi_\epsilon$ in (\ref{1.4}). Recently, we have encountered an interesting approach that constructs ambiguity sets for DRO from a Bayesian statistical perspective, as seen in \cite{bdro,bdrowzw, bdro2}. These studies demonstrate that the usual DRO methodology combined with Bayesian statistics can construct ambiguity sets which have good theoretical guarantees (e.g., finite sample guarantees and asymptotic convergence) and are applicable to a wide range of practical problems.

Inspired by these works, and to overcome the aforementioned shortcomings of existent studies based on statistics, we focus on a new method that combines prior knowledge with data to better model and solve DRCC problems. To the best of our knowledge, all existing studies on DRCC are based on classical statistics. Moreover,  we will provide a unified treatment of different distribution types, and establish finite sample guarantees and asymptotic convergence by relating the existing literature with Bayesian statistics. We not only consider discrete distributions and continuous distributions, but also account for different correlations between marginal distributions, which are rarely examined in existing DRCC literature.

From the Bayesian statistical perspective, we propose a novel approach to tackle the DRCC problem by incorporating prior knowledge and statistical inference. Our approach, called the Bayesian distributionally robust chance constraint (BDRCC), constructs a data-driven ambiguity set for the  underlying distribution of the random parameter. Concretely, for the random variable $\tilde{\pmb{\xi}}$ whose distribution $\mathbb{P}^c$ is not exactly known, we seek to configure an uncertainty set $\Xi_\epsilon$ that satisfies desirable properties. Unlike previous researches, we assume that the true distribution of $\tilde{\pmb{\xi}}$ belongs to a parameterized family of distributions $\mathcal{P}_\Theta:=\{\mathbb{P}_{\pmb{\theta}},\pmb{\theta}\in\Theta\}$, where $\Theta$ is the parameter space and $\pmb{\theta}^c\in\Theta$ is the unknown true parameter corresponding to $\mathbb{P}_{\pmb{\theta}^c}$. Recall that through the Bayesian perspective, $\pmb{\theta}^c$ can be viewed as a realization of a belief random variable $\tilde{\pmb{\theta}}$, whose posterior distribution can be easily computed based on the available data. To quantify what we learn about $\mathbb{P}_{\pmb{\theta}^c}$ from the data and ensure the tractability, we adopt a Bayesian credible interval \cite{bayes,BDA} to construct ambiguity sets that cover the true distribution with high confidence. Specifically, we assume that we have access to a sample set ${\mathcal{S}^{N}} = \{\pmb{\xi}^1,\cdots,\pmb{\xi}^N\}$ drawn i.i.d. according to $\mathbb{P}_{\pmb{\theta}^c}$. In our paradigm, we concentrate on the following data-driven robust counterpart of (\ref{1.4}):
\begin{equation}
	\text{if }g({\pmb{\xi}},\pmb{x}^*)\leq0,\ \forall \pmb{\xi}\in\Xi_\epsilon,\ \text{then }\inf_{\pmb{\theta}\in\Theta_{\mathcal{S}^{N}}}\mathbb{P}_{\pmb{\theta}}(g(\tilde{\pmb{\xi}},\pmb{x}^*)\leq0)\geq1-\epsilon.\label{1.5}
\end{equation}
Based on the prior knowledge about the parametric distribution and the sample set ${\mathcal{S}^{N}}$, the subset $\Theta_{\mathcal{S}^{N}}$ of $\Theta$ can capture various features of $\mathbb{P}_{\pmb{\theta}^c}$. Through (\ref{1.5}), the BDRCC can be reformulated as a RO constraint with uncertainty sets that imply similar probabilistic guarantees, but are much smaller than their traditional counterparts. As a result, RO models constructed using these new uncertainty sets can generate less conservative solutions than traditional ones, while still maintaining their robustness.

The significance of our approach lies in its ability to convert the BDRCC problem into a RO problem with a convex uncertainty set. Through the parametric distribution assumption, we generalize the uncertainty set construction method for DRCC problems under restrictive distribution assumptions (e.g., normal distribution).  This allows us to incorporate both Bayesian prior and posterior knowledge of the distribution into the construction of an uncertainty set, as well as to address the challenge of constructing ambiguity sets in the traditional distributionally robust framework. Our approach to DRCC problem employs a Bayesian posterior distribution to construct the ambiguity set of distributions, providing both finite sample guarantees and asymptotic convergence. By utilizing the asymptotic property of the Bayesian posterior distribution, our method  results in an improved performance compared to the classic DRCC approach. By deriving the credible intervals of the parameters corresponding to the true distribution and constructing an ambiguity set for the distribution parameters, our study establishes a direct relationship between the ambiguity set of parameters $\pmb{\theta}$ in BDRCC and the uncertainty set for the random vector $\tilde{\pmb{\xi}}$ in RO.

In this paper, we use standard notation to denote vectors, scalars, and sets, where lowercase bold letters $(\pmb{x},\pmb{\xi},...)$ represent vectors and ordinary lowercase letters $(x,\xi)$ represent scalars. Calligraphic type $(\mathcal{P}, {\mathcal{S}^{N}},...)$ denotes sets. We denote the $i^{th}$ coordinate vector as $\pmb{e}_i$ and the vector of all ones as $\pmb{e}$. A random vector is denoted by $\tilde{\pmb{\xi}}$, with its components denoted by $\tilde{\xi}_i$. We use ${\mathbb{P}_{\pmb{\theta}}}$ to denote the parametric family of distributions, and $\mathbb{P}_{\pmb{\theta}^c}$ to denote the true distribution of $\tilde{\pmb{\xi}}$. The marginal distribution of $\tilde{\xi}_i$ is denoted by $\mathbb{P}_{\pmb{\theta}_i}$. Let ${\mathcal{S}^{N}} = \{\pmb{\xi}^1,\cdots,\pmb{\xi}^N\}$ represent a sample of $N$ data points drawn i.i.d. according to $\mathbb{P}_{\pmb{\theta}^c}$, and let $\mathbb{P}_{{\mathcal{S}^{N}}}$  denote the posterior distribution based on ${\mathcal{S}^{N}}$. Finally, we denote the posterior mode of $\mathbb{P}_{{\mathcal{S}^{N}}}$ with respect to (w.r.t.) ${\mathcal{S}^{N}}$ as $\hat{\pmb{\theta}}$.

The remainder of this paper is structured as follows. In the next section, we review some relevant results that will be used throughout the subsequent discussion. Section 3 proposes an approximation to the BDRCC using the RO approach. In Sections 4 and 5, we construct various data-driven uncertainty sets that can be used under different situations. Finally, we present the numerical results of two applications in Section 6, and we conclude the paper in Section 7.

\section{Prerequisites}
In this section, we will present some preparatory knowledge that will be valuable in obtaining a tractable reformulation of the BDRCC in (\ref{1.5}). Our analysis is based on two standard assumptions:
\begin{assumption}\label{ass2.1}
The true distribution of the random parameter $\tilde{\pmb{\xi}}$ belongs to a family of distributions $\{\mathbb{P}_{\pmb{\theta}}, \pmb{\theta}\in\Theta\subseteq\mathbb{R}^n\}$ which is specified by a finite-dimensional parameter vector $\pmb{\theta}\in\mathbb{R}^n$. The true distribution corresponds to an underlying parameter ${\pmb{\theta}^c}$, i.e., $\mathbb{P}^c=\mathbb{P}_{\pmb{\theta}^c}$  for some $\pmb{\theta}^c\in\Theta\subseteq\mathbb{R}^n$.
\end{assumption}

\begin{assumption}\label{ass2.2}
     The constraint function $ g (\cdot,\pmb{x })$ is concave and finite-valued in $\pmb{\xi}$ for all $\pmb{x}$. 
\end{assumption}

As a Bayesian perspective assumption, Assumption \ref{ass2.1} is general enough to encompass typical cases. For example, it includes general parametric distributions such as the Poisson and normal distributions. It also covers finite discrete distributions and finite mixture distributions with known components.

Assumption \ref{ass2.2} is quite mild compared to those in the existing literature. Most typical DRCC problems are studied under linear constraints, which automatically satisfy Assumption \ref{ass2.2}.
\subsection{Tractability of Distributionally Robust Chance Constraints}
For presentation convenience, we first recall the definitions of the support function and the Value at Risk (VaR). For any $\pmb{v}\in\mathbb{R}^d$, $0<\epsilon<1$ and probability distribution $\mathbb{P}$, let $\delta^*(\pmb{v}|\Xi)$ denote the support function of $\Xi$ and $\text{VaR}_\epsilon^\mathbb{P}(\pmb{v}^T\tilde{\pmb{\xi}})$ denote the VaR at level $\epsilon$ with respect to $\pmb{v}$, defined respectively as
\[\delta^*(\pmb{v}|\Xi)\equiv\sup_{\pmb{\xi}\in\Xi}\pmb{v}^T\pmb{\xi},\ \text{VaR}_\epsilon^\mathbb{P}(\pmb{v}^T\tilde{\pmb{\xi}})\equiv\inf\{t:\mathbb{P}(\pmb{v}^T\tilde{\pmb{\xi}}\leq t)\geq1-\epsilon\}.\] 

As a generalization of Theorem 1 in \cite{main}, the theorem below demonstrates that the tractability of the probability guarantee problem under BDRCC depends on the relationship between an upper bound of VaR over $\Theta_{\mathcal{S}^N}$ and the support function of the uncertainty set $\Xi_\epsilon$. 

\begin{theorem}\label{thm2.1}
    Under Assumptions \ref{ass2.1} and \ref{ass2.2}, if the non-empty, convex, and compact uncertainty set $\Xi_\epsilon$ satisfies  $\delta^*(\pmb{v}|\Xi_\epsilon)\geq\sup_{\pmb{\theta}\in\Theta_{\mathcal{S}^N}}\text{VaR}_\epsilon^{\mathbb{P}_{\pmb{\theta}}}(\pmb{v}^T\tilde{\pmb{\xi}})$  for every $\pmb{v}\in\mathbb{R}^d$, then for any $\pmb{x}^*\in\mathbb{R}^k$, the constraint 
\[g({\pmb{\xi}},\pmb{x}^*)\leq0,\ \forall \pmb{\xi}\in\Xi_\epsilon,\]
implies the BDRCC
\begin{equation}
	\inf_{\pmb{\theta}\in\Theta_{\mathcal{S}^N}}\mathbb{P}_{\pmb{\theta}}(g(\tilde{\pmb{\xi}},\pmb{x}^*)\leq0)\geq1-\epsilon.\label{BDRCC}
\end{equation}
\end{theorem}\par
 \begin{proof}
      For any $\pmb{x}^*\in\mathbb{R}^k$ that satisfies $g(\pmb{\xi},\pmb{x}^*)\leq0,\ \forall\pmb{\xi}\in\Xi_\epsilon$, it is obvious that $\max_{\pmb{\xi}\in\Xi_\epsilon}g(\pmb{\xi}, \pmb{x}^*) \leq 0$.
Define a closed and convex set $\Xi_t:=\{\pmb{\xi}\in\mathbb{R}^d: g (\pmb{\xi}, \pmb{x}^*)\geq t\}$ with $t > 0$. It is obvious that the intersection of $\Xi_\epsilon$ and $\Xi_t$ is empty.  According to the separating hyperplane theorem,  $\Xi_\epsilon$ and $\Xi_t$ can be separated by a strict separating hyperplane $\pmb{v}^T\pmb{\xi} = {v}_0$ such that $\pmb{v}^T\pmb{\xi}_1 >{v}_0 > \pmb{v}^T\pmb{\xi}_2$ for all $\pmb{\xi}_1\in\Xi_t$ and $\pmb{\xi}_2\in\Xi_\epsilon$. Thus $v_0$ satisfies
\[{v}_0>\max_{\pmb{\xi}\in\Xi_\epsilon}\pmb{v}^T\pmb{\xi}=\delta^*(\pmb{v}|\Xi_\epsilon)\geq\sup_{\mathbb{P}\in\mathcal{P}_{\Theta_{\mathcal{S}^N}}}\text{VaR}_\epsilon^{\mathbb{P}}(\pmb{v}^T\tilde{\pmb{\xi}})=\sup_{\pmb{\theta}\in\Theta_{\mathcal{S}^N}}\text{VaR}_\epsilon^{\mathbb{P}_{\pmb{\theta}}}(\pmb{v}^T\tilde{\pmb{\xi}}),\]
which leads to 
\[\mathbb{P}_{\pmb{\theta}}(g(\tilde{\pmb{\xi}},\pmb{x}^*)\geq t)\leq\mathbb{P}_{\pmb{\theta}}(\pmb{v}^T\tilde{\pmb{\xi}}>{v}_0)\leq\mathbb{P}_{\pmb{\theta}}(\pmb{v}^T\tilde{\pmb{\xi}}>\text{VaR}_\epsilon^{\mathbb{P}_{\pmb{\theta}}}(\pmb{v}^T\tilde{\pmb{\xi}}))\leq\epsilon,\ \forall\pmb{\theta}\in\Theta_{\mathcal{S}^N}.\]
Letting $t\to0$ gives us that $\mathbb{P}_{\pmb{\theta}}(g(\tilde{\pmb{\xi}},\pmb{x}^*)>0)\leq\epsilon$ for all $\pmb{\theta}\in\Theta_{\mathcal{S}^N}$.  Thus we can conclude that
\[	\text{if }g({\pmb{\xi}},\pmb{x}^*)\leq0,\ \forall \pmb{\xi}\in\Xi_\epsilon,\ \text{then }\inf_{\pmb{\theta}\in\Theta_{\mathcal{S}^N}}\mathbb{P}_{\pmb{\theta}}(g(\tilde{\pmb{\xi}},\pmb{x}^*)\leq0)\geq1-\epsilon.\]
That is, $\Xi_\epsilon$ implies the BDRCC (\ref{BDRCC}).
 \end{proof}

This theorem paves the way for the construction of a new type of data-driven uncertainty sets later on.
\subsection{Credible Interval for Parametric Distributions}
To construct a reasonable and implementable uncertainty set for $\pmb{\theta}$ which includes the true parameter $\pmb{\theta}^c$ of the underlying probability distribution with high probability, we will briefly review some needed results about credible interval based on a parameterized family of distributions. 

Under the parametric assumption on $\mathbb{P}^c$, the probability density function (pdf) $p(\cdot|\pmb{\theta})$ is also defined with respect to the parameter $ \pmb{\theta} $.  Different from the existing literature, here $\pmb{\theta}$ is treated as a random variable which lies in the set $\Theta$ and follows a prior pdf $p(\pmb{\theta})$. With a given sample set ${\mathcal{S}^{N}} = \{\pmb{\xi}^1,\cdots,\pmb{\xi}^N\}$, the pdf of the posterior distribution $\mathbb{P}_{\mathcal{S}^N}$ is determined, according to the Bayes’ rule, as
\begin{equation}
	p(\pmb{\theta}|{\mathcal{S}^{N}})=\frac{p(\mathcal{S}^N|\pmb{\theta})p(\pmb{\theta})}{\int_\Theta p(\mathcal{S}^N|\pmb{\theta})p(\pmb{\theta})d\pmb{\theta}},\label{new2.1}
\end{equation}
where $p(\mathcal{S}^N|\pmb{\theta})=\prod_{i=1}^Np(\pmb{\xi}^i|\pmb{\theta})$ is the conditional pdf of the sample data.

Based on the posterior distribution $p(\pmb{\theta}|{\mathcal{S}^{N}})$, we can construct an interval which covers the true value of the parameter with a high probability. For example, a 95\% central posterior interval for $\pmb{\theta}$ will cover its true value 95\% of the time when sampling is repeated with respect to the true $\pmb{\theta}^c$. This interval based on Bayesian posterior distribution is called the credible interval  \cite{bayes,BDA}, which will play a critical role in our hereinafter construction. If the Bayesian credible interval has a closed form, we can easily construct a set  $\Theta(\mathcal{S}^N,\alpha)$ such that $\mathbb{P}_{\mathcal{S}^N}(\mathbb{P}_{\pmb{\theta}^c}\in\{\mathbb{P}_{\pmb{\theta}}|\pmb{\theta}\in{\Theta}({\mathcal{S}^{N}},\alpha)\})=1-\alpha$ for any given $0\leq\alpha<1$. 

For the general case, we can consider a quadratic approximation to the log-posterior density function $ \log p(\pmb{\theta}|{\mathcal{S}^{N}}) $ that is centered at the posterior mode (which in general is easy to compute using off-the-shelf optimization routines). Concretely, a Taylor series expansion of $ \log p(\pmb{\theta}|{\mathcal{S}^{N}}) $ centered at the posterior mode, $\hat{\pmb{\theta}}$, which is assumed to be in the interior of the parameter space, gives
\begin{equation}
	\log p(\pmb{\theta}|{\mathcal{S}^{N}})=\log p(\hat{\pmb{\theta}}|{\mathcal{S}^{N}})+\frac{1}{2}(\pmb{\theta}-\hat{\pmb{\theta}})^T\left[\frac{d^2}{d\pmb{\theta}^2}\log p(\pmb{\theta}|{\mathcal{S}^{N}})\right]_{\pmb{\theta}=\hat{\pmb{\theta}}}(\pmb{\theta}-\hat{\pmb{\theta}})+\cdots,\label{2.1}
\end{equation}
where the linear term in the expansion is zero because the log-posterior density has zero derivative at its mode.

\begin{lemma}(Gelman et al.\cite{BDA})\label{lem2.1}
Suppose that the pdf of the posterior distribution $\mathbb{P}_{\mathcal{S}_N}$ is a continuous function of $\pmb{\theta}$, and $\pmb{\theta}^c$ is not on the boundary of $\Theta$,  the logarithm of the posterior density (\ref{2.1}) can be approximated by a quadratic function of $\pmb{\theta}$, yielding the following approximation:
\begin{equation}
	p(\pmb{\theta}|{\mathcal{S}^{N}})\approx N(\hat{\pmb{\theta}},	[I(\hat{\pmb{\theta}})]^{-1}).\label{2.2}
\end{equation}
Here $I(\pmb{\theta})$ is the observed information matrix defined as
\[I (\pmb{\theta})=-\frac{d^2}{d\pmb{\theta}^2}\log p(\pmb{\theta}|{\mathcal{S}^{N}}).\]
Furthermore, if the mode $\hat{\pmb{\theta}}$ is in the interior of the parameter space, $I(\hat{\pmb{\theta}})$ is positive definite.
\end{lemma}
With the given data ${\mathcal{S}^{N}}$, we can determine the posterior ${\hat{\pmb{\theta}}}$ and the observed information matrix $I(\hat{\pmb{\theta}})$. The Bayesian credible interval $\Theta({\mathcal{S}^{N}},\alpha)$ can then be defined as $\{\pmb{\theta}\in\Theta:||[I(\hat{\pmb{\theta}})]^{1/2} (\pmb{\theta} - \hat{\pmb{\theta}})||_2\leq z_{1-\frac{\alpha}{2}}\}$, where $z_{1-\frac{\alpha}{2}}$ is the ${1-\frac{\alpha}{2}}$ quantile of the standard normal distribution. This credible interval guarantees that we will incorrectly reject that $\pmb{\theta}^c\in\bar{\Theta}({\mathcal{S}^{N}},\alpha):=\Theta\backslash\Theta({\mathcal{S}^{N}},\alpha)$ with probability at most $\alpha$ with respect to ${\mathcal{S}^{N}}$. The Bayesian credible interval therefore leads to the following (approximate) $(1 - \alpha)$ credible set around  $\mathbb{P}_{\hat{\pmb{\theta}}}$ w.r.t. the posterior distribution  $\mathbb{P}_{\mathcal{S}^N}$:
\[\left\{\mathbb{P}_{\pmb{\theta}}|\pmb{\theta}\in\Theta({\mathcal{S}^{N}},\alpha)\right\}\text{ which satisfies }\mathbb{P}_{\mathcal{S}^N}(\mathbb{P}_{\pmb{\theta}^c}\in\{\mathbb{P}_{\pmb{\theta}}|\pmb{\theta}\in{\Theta}({\mathcal{S}^{N}},\alpha)\})=1-\alpha.\]
That is, we will incorrectly reject that $\mathbb{P}_{\pmb{\theta}^c}\in\{\mathbb{P}_{\pmb{\theta}}|\pmb{\theta}\in\bar{\Theta}({\mathcal{S}^{N}},\alpha)\}$ at most $\alpha$ w.r.t. ${\mathcal{S}^{N}}$.

In this way, the probability (w.r.t. $\mathbb{P}_{\mathcal{S}^{N}}$) that $\mathbb{P}_{\pmb{\theta}^c}$ belongs to our credible interval is at least $1 - \alpha$. Therefore, different from current researches, despite not knowing $\mathbb{P}_{\pmb{\theta}^c}$, we can use a Bayesian credible interval to create a set of distributions from the data that can contain $\mathbb{P}_{\pmb{\theta}^c}$ for any specified probability.

\begin{theorem}
For the fixed uncertainty set $\Theta({\mathcal{S}^{N}},\alpha)\}$ for the parameter $\pmb{\theta}$, the BDRCC 

\begin{equation}
\inf_{\pmb{\theta}\in\Theta({\mathcal{S}^{N}},\alpha)}\mathbb{P}_{\pmb{\theta}}\left(g(\tilde{\pmb{\xi}},\pmb{x})\leq0\right)\geq1-{\epsilon}\label{BDRCC1}
\end{equation}
will imply the DRCC
\begin{equation}
    \mathbb{P}_{\pmb{\theta}^c}\left(g(\tilde{\pmb{\xi}},\pmb{x})\leq0\right)\geq1-{\epsilon}
\label{drcc1}
\end{equation}
for $\mathbb{P}_{\pmb{\theta}^c}$ with probability at least $1-\alpha$ with respect to $\mathbb{P}_{\mathcal{S}^{N}}$.
\end{theorem}
\begin{proof}
It is obviously since
    \begin{equation}
	\nonumber
		\mathbb{P}_{\mathcal{S}^{N}}\left(\text{the BDRCC }(\ref{BDRCC1}) \text{ implies DRCC} (\ref{drcc1})\right)
		\geq\mathbb{P}_{\mathcal{S}^{N}}(\pmb{\theta}^c\in\Theta({\mathcal{S}^{N}},\alpha))
		\geq1-\alpha.
\end{equation}
\end{proof}
\section{Construction of the data-driven uncertainty set}
With the Bayesian credible interval defined above, we can configure an ambiguity set $\mathcal{P}_{\Theta({\mathcal{S}^{N}},\alpha)}=\{\mathbb{P}_{\pmb{\theta}}|\pmb{\theta}\in\Theta({\mathcal{S}^{N}},\alpha)\}$ which  will contain $\mathbb{P}_{\pmb{\theta}^c}$ with probability at least $1 - \alpha$ for any $0\leq\alpha<1$. With this and Theorem \ref{thm2.1}, we can then find an uncertainty set $\Xi({\mathcal{S}^{N}},\epsilon,\alpha)$ which will ensure the relevant BDRCC (\ref{BDRCC1}). The key here is to construct a support function of $\Xi({\mathcal{S}^{N}},\epsilon,\alpha)$ which satisfies $\delta^*(\pmb{v}|\Xi({\mathcal{S}^{N}},\epsilon,\alpha))\geq\sup_{\pmb{\theta}\in\Theta_{\mathcal{S}^N}}\text{VaR}_\epsilon^{\mathbb{P}_{\pmb{\theta}}}(\pmb{v}^T\tilde{\pmb{\xi}})$, and we accomplish this by finding a positive homogeneity type of upper bound of VaR. The concrete scheme is stated as the following Algorithm 1.
	\begin{algorithm}[h]  
	\caption{Constructing an uncertainty set $\Xi({\mathcal{S}^{N}},\epsilon,\alpha)$}  \label{algo1}
	\hspace*{0.02in}{\bf Input:}
	the sample set ${\mathcal{S}^{N}}=\{\pmb{\xi}^1,\cdots,\pmb{\xi}^{N}\}$,  the prior distribution of $\tilde{\pmb{\xi}}$, the probability guarantee level $\epsilon\in(0,1)$, and the credible level $\alpha\in[0,1)$.\\
	\hspace*{0.02in}{\bf Output:} 
	the uncertainty set $\Xi({\mathcal{S}^{N}},\epsilon,\alpha)$. \\
 $\bullet$ Calculate the exact posterior distribution $\mathbb{P}_{\mathcal{S}^{N}}$ by (\ref{new2.1}) or estimate $\mathbb{P}_{\mathcal{S}^{N}}$ through the approximation (\ref{2.2}), and then configure the Bayesian credible interval $\Theta({\mathcal{S}^{N}},\alpha)$.\\
	$\bullet$ For chosen $\pmb{\theta}$ and $\epsilon$, identify a positive homogenous upper bound to VaR as follows:
  \[\text{VaR}_\epsilon^{\mathbb{P}_{\pmb{\theta}}}(\pmb{v}^T\tilde{\pmb{\xi}})\leq\pmb{v}^T\pmb{B}(\pmb{\theta},\epsilon).\]
  $\bullet$ Find a closed and convex set $\Xi({\mathcal{S}^{N}},\epsilon,\alpha)$ such that $\sup_{\pmb{\theta}\in\Theta({\mathcal{S}^{N}},\alpha)}\pmb{v}^T\pmb{B}(\pmb{\theta},\epsilon)=
  \delta^* (\pmb{v}|\Xi({\mathcal{S}^{N}},\epsilon,\alpha))$.
\end{algorithm}

\begin{theorem}\label{thm3.1}
    For any $0<\epsilon<1$, the uncertainty set $\Xi({\mathcal{S}^{N}},\epsilon,\alpha)=\{\pmb{B}(\pmb{\theta},\epsilon),\pmb{\theta}\in\Theta({\mathcal{S}^{N}},\alpha)\}$ implies the BDRCC (\ref{BDRCC1}) which induce DRCC (\ref{drcc1})
    for $\mathbb{P}_{\pmb{\theta}^c}$ with probability at least $1-\alpha$ with respect to $\mathbb{P}_{\mathcal{S}^{N}}$.
\end{theorem}
\begin{proof}
Due to Theorem \ref{thm2.1}, we can ensure that the closed and convex set $\Xi({\mathcal{S}^{N}},\epsilon,\alpha)$ constructed by Algorithm \ref{algo1} will satisfy 
\[	\text{if }g({\pmb{\xi}},\pmb{x}^*)\leq0,\ \forall \pmb{\xi}\in\Xi({\mathcal{S}^{N}},\epsilon,\alpha),\ \text{then }\inf_{\pmb{\theta}\in\Theta({\mathcal{S}^{N}},\alpha)}\mathbb{P}_{\pmb{\theta}}(g(\tilde{\pmb{\xi}},\pmb{x}^*)\leq0)\geq1-\epsilon.\] Moreover, from the definition of the support function, we can easily obtain the closed form of the uncertainty set as $\Xi({\mathcal{S}^{N}},\epsilon,\alpha)=\{\pmb{B}(\pmb{\theta},\epsilon),\pmb{\theta}\in\Theta({\mathcal{S}^{N}},\alpha)\}$ in our schema. We now demonstrate the properties of the constructed uncertainty set for all $0<\epsilon<1$.
    \begin{equation}
	\nonumber
	\begin{split}
		&\mathbb{P}_{\mathcal{S}^{N}}\left(\Xi({\mathcal{S}^{N}},\epsilon,\alpha)\text{ implies the DRCC in }(\ref{drcc1}) \text{ for $\mathbb{P}_{\pmb{\theta}^c}$}, \ 0<\epsilon<1 \right) \\
		=&\mathbb{P}_{\mathcal{S}^{N}}\left(\delta^*(\pmb{v}|\Xi({\mathcal{S}^{N}},\epsilon,\alpha))\geq\text{VaR}_\epsilon^{\mathbb{P}_{\pmb{\theta}^c}}(\pmb{v}^T\tilde{\pmb{\xi}}),\ \forall \pmb{v}\in\mathbb{R}^d, \ 0<\epsilon<1 \right) \\
		=&\mathbb{P}_{\mathcal{S}^{N}}\left(\sup_{\pmb{\theta}\in\Theta({\mathcal{S}^{N}},\alpha)}\pmb{v}^T\pmb{B}(\pmb{\theta},\epsilon)\geq\text{VaR}_\epsilon^{\mathbb{P}_{\pmb{\theta}^c}}(\pmb{v}^T\tilde{\pmb{\xi}}),\ \forall \pmb{v}\in\mathbb{R}^d,\ 0<\epsilon<1 \right)  \\
		\geq&\mathbb{P}_{\mathcal{S}^{N}}\left(\pmb{\theta}^c\in\cap_{\epsilon:0<\epsilon<1}\Theta({\mathcal{S}^{N}},\alpha)\right)\\
		=&\mathbb{P}_{\mathcal{S}^{N}}(\pmb{\theta}^c\in\Theta({\mathcal{S}^{N}},\alpha))\\
		\geq&1-\alpha.
	\end{split}
\end{equation}
The first equation is induced by Theorem \ref{thm2.1} with $\Theta_{\mathcal{S}^{N}}$ being viewed as the uncertainty set of $\pmb{\theta}^c$ and the second equation is obvious. The first inequality is based on our algorithm. With the definition of credible interval, $\Theta({\mathcal{S}^{N}},\alpha)$ is independent of $\epsilon$. Thus, the uncertainty set $\Xi({\mathcal{S}^{N}},\epsilon,\alpha)$ will  imply the DRCC in (\ref{drcc1}) for $\mathbb{P}_{\pmb{\theta}^c}$ with probability at least $1-\alpha$ w.r.t. $\mathbb{P}_{\mathcal{S}^{N}}$.
\end{proof}

The following theorem  states that when the support set of a random variable is known, the corresponding BDRCC can be approximated with the intersection of the support set and our constructed uncertainty set.

\begin{theorem}\label{thm3.2}
    Suppose supp$(\mathbb{P}_{\pmb{\theta}^c})\subseteq\Xi_0$ with $\Xi_0$ being closed and convex. If  for all $0<\epsilon<1$, the uncertainty set $\Xi({\mathcal{S}^{N}},\epsilon,\alpha)=\{\pmb{B}(\pmb{\theta},\epsilon),\pmb{\theta}\in\Theta({\mathcal{S}^{N}},\alpha)\}$ implies the BDRCC (\ref{BDRCC1}) for $\mathbb{P}_{\pmb{\theta}^c}$, then  $\Xi({\mathcal{S}^{N}},\epsilon,\alpha)\cap\Xi_0$ also implies the BDRCC (\ref{BDRCC1}).
\end{theorem}

\textbf{Proof}. It can be seen from Theorem \ref{thm2.1} that $\sup_{\pmb{\theta}\in\Theta({\mathcal{S}^{N}},\alpha)}\text{VaR}^{\mathbb{P}_{\pmb{\theta}}}_\epsilon(\pmb{v}^T\tilde{\pmb{\xi}})\leq\delta^*(\pmb{v}|\Xi({\mathcal{S}^{N}},\epsilon,\alpha))$ for all $\pmb{v}$. What's more, since supp$(\mathbb{P}_{\pmb{\theta}})\subseteq\Xi_0$, we have 
\begin{equation}
\begin{split}
    	0 = &\sup_{\pmb{\theta}\in\Theta\left({\mathcal{S}^{N}},\alpha\right)}\mathbb{P}_{\pmb{\theta}}(\pmb{v}^T \tilde{\pmb{\xi}} > \max_{\pmb{\xi}\in \text{supp}(\mathbb{P}_{\pmb{\theta}^c})} \pmb{v}^T \xi)\\
     \geq &\sup_{\pmb{\theta}\in\Theta({\mathcal{S}^{N}},\alpha)}\mathbb{P}_{\pmb{\theta}}(\pmb{v}^T\tilde{\pmb{\xi}} > \max_{\pmb{\xi}\in\Xi_0}\pmb{v}^T\pmb{\xi}) = \sup_{\pmb{\theta}\in\Theta({\mathcal{S}^{N}},\alpha)}\mathbb{P}_{\pmb{\theta}}(\pmb{v}^T\tilde{\pmb{\xi}} > \delta^* (\pmb{v}|\Xi_0)).
\end{split}
\label{2.6}
\end{equation}
This implies $\sup_{\pmb{\theta}\in\Theta({\mathcal{S}^{N}},\alpha)}\text{VaR}_\epsilon^{\mathbb{P}_{\pmb{\theta}}}(\pmb{v}^T\tilde{\pmb{\xi}})\leq\delta^*(\pmb{v}|\Xi_0)$. Then we have \[\sup_{\pmb{\theta}\in\Theta({\mathcal{S}^{N}},\alpha)}\text{VaR}_\epsilon^{\mathbb{P}_{\pmb{\theta}}}(\pmb{v}^T\tilde{\pmb{\xi}})\leq\min(\delta^*(\pmb{v}|\Xi({\mathcal{S}^{N}},\epsilon,\alpha))\ \delta^*(\pmb{v}|\Xi_0))=\delta^*(\pmb{v}|\Xi_0\cap\Xi({\mathcal{S}^{N}},\epsilon,\alpha))\]
because both $\Xi_0$ and $\Xi({\mathcal{S}^{N}},\epsilon,\alpha)$ are convex sets. Thus, $\Xi({\mathcal{S}^{N}},\epsilon,\alpha)\cap\Xi_0$
implies the BDRCC (\ref{BDRCC1}) by Theorem \ref{thm2.1}. $\hfill\qedsymbol$

With Theorem \ref{thm3.1} and Theorem \ref{thm3.2}, we can conclude that the feasible decision $\pmb{x}$ for the RO problem induced from the uncertainty set $\Xi({\mathcal{S}^{N}},\epsilon,\alpha)=\{\pmb{B}(\pmb{\theta},\epsilon),\pmb{\theta}\in\Theta({\mathcal{S}^{N}},\alpha)\}$ will satisfy the corresponding BDRCC problem with a single constraint $g (\tilde{\pmb{\xi}},\pmb{x})\leq0$ at level $1-\epsilon$ with probability at least $1-\alpha$ with respect to $\mathbb{P}_{\mathcal{S}^{N}}$. Nevertheless, as quite a few studies in the literature (\cite{68,ref22,86}) show, we often need to consider the joint chance constraint case with multiple constraints $g_j(\tilde{\pmb{\xi}},\pmb{x}) \leq 0, j = 1,\cdots,J$. To address this, we can extend the idea of the BDRCC framework (\ref{1.5}) at level $1-\bar{\epsilon}$ with an ambiguity set $\Theta({\mathcal{S}^N,\alpha})$ of parameters to the following situation:
\begin{equation}
\inf_{\pmb{\theta}\in\Theta({\mathcal{S}^{N}},\alpha)}\mathbb{P}_{\pmb{\theta}}\left(\max_{j=1,\cdots,J}g_j(\tilde{\pmb{\xi}},\pmb{x})\leq0\right)\geq1-\bar{\epsilon}\label{2.3}
\end{equation}
for some given $\bar{\epsilon}$. A simple way to ensure (\ref{2.3}) is to set $\epsilon_j=\bar{\epsilon}/J$ and construct the uncertainty set $\Xi({\mathcal{S}^{N}},\epsilon_j,\alpha)$ w.r.t. the $j$-th constraint $g_j(\tilde{\pmb{\xi}},\pmb{x})\leq0$ by applying our schema, Algorithm 1. That is, we can easily approximate (\ref{2.3}) through
\begin{equation}	\inf_{\pmb{\theta}\in\Theta({\mathcal{S}^{N}},\alpha)}\mathbb{P}_{\pmb{\theta}}\left(g_j(\tilde{\pmb{\xi}},\pmb{x})\leq0\right)\geq1-\epsilon_j,\ \ {j=1,\cdots,J}.\label{new2.3}
\end{equation}
And thus the corresponding RO problem with the following $J$ constraints,
\begin{equation}
	g_j(\pmb{\xi},\pmb{x})\leq0,\ \forall\pmb{\xi}\in\Xi({\mathcal{S}^{N}},\epsilon_j,\alpha),\ j=1,\cdots,J\label{2.4}
\end{equation}
will imply the BDRCC (\ref{new2.3})  at level $1-\epsilon_j$,  respectively, w.r.t.  $\mathbb{P}_{\mathcal{S}^{N}}$. 
Finally, the RO problem with the following constraint
\[\max_{j=1,\cdots,J}\left\{\max_{\pmb{\xi}\in\Xi({\mathcal{S}^{N}},\epsilon_j,\alpha)}g_j(\pmb{\xi},\pmb{x})\right\}\leq0\]
will imply the BDRCC (\ref{2.3}) at level $1-\bar{\epsilon}$ w.r.t.  $\mathbb{P}_{\mathcal{S}^{N}}$. 

The above setting of (\ref{2.4}) only provides a trivial way to approximate (\ref{2.3}). A better way is to find optimal $\epsilon_j$, $j=1,\cdots,J$, satisfying $\sum_{j=1}^J\epsilon_j=\bar{\epsilon}$ through solving the following optimization problem w.r.t. $\epsilon_j$, $j=1,\cdots,J$:
\begin{equation}
	\min_{\epsilon_1+\cdots+\epsilon_J=\bar{\epsilon},\epsilon\geq0}\left\{\max_{j=1,\cdots,J}\left\{\max_{\pmb{\xi}\in\Xi({\mathcal{S}^{N}},\epsilon_j,\alpha)}g_j(\pmb{\xi},\pmb{x})\right\}\right\}\leq0 \label{2.5}
\end{equation}

The next theorem shows that a feasible solution to problem (\ref{2.5}) will satisfy the true DRCC 
\begin{equation}
\mathbb{P}_{\pmb{\theta}^c}\left(\max_{j=1,\cdots,J}g_j(\tilde{\pmb{\xi}},\pmb{x})\leq0\right)\geq1-\bar{\epsilon}\label{drccmulti}
\end{equation}
with probability at least $1-\alpha$ w.r.t. $\mathbb{P}_{\mathcal{S}^{N}}$ regardless of the choice of the sample set $\mathcal{S}^N$.

\begin{theorem}\label{thm3.3}
    With probability $1-\alpha$ w.r.t. $\mathbb{P}_{\mathcal{S}^{N}}$, a feasible solution $\pmb{x}$ to (\ref{2.5}) which implies BDRCC (\ref{2.3}) will satisfy the DRCC (\ref{drccmulti}).
\end{theorem}
\begin{proof}
    For any feasible $\epsilon_j,\ j=1,\cdots,J$ in (\ref{2.5}),  we have from Theorem \ref{thm3.1} that
	\begin{equation}
	\nonumber
	\begin{split}
		1-\alpha=&\mathbb{P}_{\mathcal{S}^{N}}\left(\Xi({\mathcal{S}^{N}},\epsilon,\alpha)\text{ implies the DRCC }(\ref{drcc1}),\ \text{for }\ 0<\epsilon<1\right) \\
		\leq&\mathbb{P}_{\mathcal{S}^{N}}(\Xi({\mathcal{S}^{N}},\epsilon_j,\alpha)\text{ implies the DRCC (\ref{drccmulti}) at level }1-\epsilon_j,\ j=1,\cdots,m).
	\end{split}
	\end{equation}
The second inequality holds by setting $g(\pmb{\xi},\pmb{x})=g_j(\pmb{\xi},\pmb{x})$ for $j=1\cdots,J$, respectively. Then the desired conclusion follows obviously.
\end{proof}

This theorem tells us that although the optimal $\epsilon_j,\ j=1,\cdots,m$ may depend on the sample set $\mathcal{S}^N$, our proposed schema will never result in an in-sample bias which the method in \cite{main} does.

To show the reasonability of the proposed data-driven approach, we investigate the convergence of the uncertainty set $\Xi({\mathcal{S}^{N}},\epsilon,\alpha)$ as the sample size $N$ goes to infinity. Recall that $p(\pmb{\theta})$ denotes the prior pdf. We first make the following assumptions.

\begin{assumption}\label{ass3.1}
    (i) $\ln p(\pmb{\xi}|\pmb{\theta})$, $\pmb{\theta}\in\Theta$, is dominated by an integrable  function with respect to $\mathbb{P}_{\pmb{\theta}^c}$, (ii) there exist constants $c_1 > c_2 > 0 $ such that $c_1 \geq p(\pmb{\theta}) \geq c_2,\ \forall\pmb{\theta}\in\Theta$,  (iii) the parameter space $\Theta$ is a nonempty, compact and convex set.
\end{assumption}
Assumption \ref{ass3.1} (i) is a regular assumption and the rest are necessary for establishing the uniform convergence of the posterior distribution. Our convergence result relys on the following conclusion.
\begin{lemma} (Theorem 3.1 in \cite{BDRO})\label{lem3.1}
    Suppose Assumption \ref{ass3.1} holds. Then for $0 < b < a < \omega<+\infty$, there exists an $N$ large enough  such that
\begin{equation}
	\sup_{\pmb{\theta}\in V_\omega}p(\pmb{\theta}| {\mathcal{S}^{N}} ) \leq \kappa(b)^{-1} e^ {-N(a-b)}\label{2.7}
\end{equation}
holds  w.p.1, where  $V_\omega:= \{\pmb{\theta}\in\Theta : \phi(\pmb{\theta}^c) - \phi(\pmb{\theta}) \geq \omega\}$, $U_\omega:= \Theta \backslash V_\omega =  \{\pmb{\theta}\in\Theta : \phi(\pmb{\theta}^c) - \phi(\pmb{\theta}) < \omega\}$ and $\kappa(b):=\int_{U_b}d\pmb{\theta}$. Here $\phi(\pmb{\theta}) := \mathbb{E}_{\mathbb{P}_{\pmb{\theta}^c}}  \ln p(\pmb{\xi}|\pmb{\theta})$.
\end{lemma}

Obviously, we can see that $U_\omega$ is a neighborhood of $\pmb{\theta}^c$ and Assumption \ref{ass3.1} (iii) ensures that $\int_{U_\omega}d\pmb{\theta}>0$ for any $\omega > 0$. Lemma \ref{lem3.1} proposes that the posterior pdf $p(\pmb{\theta}| {\mathcal{S}^{N}} )$ uniformly converges to the Dirac function $\delta(\pmb{\theta}^c)$ w.p.1, no matter what the prior pdf $p(\pmb{\theta})$ we choose. In the following, w.p.1 means that the considered property holds with probability one w.r.t. the probability distribution for the sequence $\{\pmb{\xi}^1,\pmb{\xi}^2,\cdots\}$ as $\mathbb{P}_{\pmb{\theta}^c}^\infty:=\mathbb{P}_{\pmb{\theta}^c}\times\mathbb{P}_{\pmb{\theta}^c}\cdots$.

\begin{theorem}\label{thm3.4}
    Suppose that Assumption \ref{ass3.1} holds. Then $\mathbb{P}_{\mathcal{S}^{N}}(\Theta({\mathcal{S}^{N}},\alpha)\subset U_\omega)$ converges to one w.p.1.
\end{theorem}
\begin{proof}
     Let $\tilde{\pmb{\theta}}$ be a random variable with the posterior pdf $p(\pmb{\theta}| {\mathcal{S}^{N}})$. The probability of the event $\{\pmb{\theta}\in V_\omega\}$ is given by the integral $\int_{V_\omega}p(\pmb{\theta}|{\mathcal{S}^{N}})d\pmb{\theta}$. Consequently under Assumption \ref{ass3.1}, we have by (\ref{2.7}) that for any $\omega > 0$ and  for $N$ large enough,
\begin{equation}
		\mathbb{P}_{{\mathcal{S}^{N}}}({\tilde{\pmb{\theta}}\in V_\omega},\tilde{\pmb{\theta}}\in\Theta({\mathcal{S}^{N}},\alpha)) \leq\mathbb{P}_{{\mathcal{S}^{N}}}({\tilde{\pmb{\theta}}\in V_\omega}) \leq  \kappa(b)^{-1}\nu e^ {-N(a-b)} \label{2.8}
\end{equation}
holds w.p.1 and 
\begin{equation}
	\mathbb{P}_{\mathcal{S}^{N}}(\tilde{\pmb{\theta}}\in U_\omega|\tilde{\pmb{\theta}}\in\Theta({\mathcal{S}^{N}},\alpha))=\frac{\mathbb{P}_{\mathcal{S}^{N}}({\tilde{\pmb{\theta}}\in U_\omega},\tilde{\pmb{\theta}}\in\Theta({\mathcal{S}^{N}},\alpha))}{\mathbb{P}_{\mathcal{S}^{N}}(\tilde{\pmb{\theta}}\in\Theta({\mathcal{S}^{N}},\alpha))}\geq1-\frac{\kappa(b)^{-1}\nu e^{-N(a-b)}}{\mathbb{P}_{\mathcal{S}^{N}}(\tilde{\pmb{\theta}}\in\Theta({\mathcal{S}^{N}},\alpha))},\label{2.9}
\end{equation}
where $\nu=\int_{\Theta}d\pmb{\theta}$ is the volume of $\Theta$. It follows that for any $\pmb{\theta}\in\Theta({\mathcal{S}^{N}},\alpha)$, the probability of the event $\{\tilde{\pmb{\theta}}\in U_\omega\}$ converges to one (exponentially fast) w.p.1 as $N \to\infty$.
\end{proof}

\begin{remark}\label{remark3.1}
    Under Assumption \ref{ass3.1}, the function $\phi: \Theta\to\mathbb{R}$ is real-valued. Moreover, we have that for $\pmb{\theta}\in\Theta$,
\[\lim_{\pmb{\theta}'\to\pmb{\theta}}\phi(\pmb{\theta}')=\lim_{\pmb{\theta}'\to\theta}\int\ln p(\pmb{\xi}|\pmb{\theta}')\mathbb{P}_{\pmb{\theta}^c}(d\xi)=\int\lim_{\pmb{\theta}'\to\pmb{\theta}}\ln p(\pmb{\xi}|\pmb{\theta}')\mathbb{P}_{\pmb{\theta}^c}(d\xi)=\phi(\pmb{\theta}),\]
where we use the continuity of $p(\pmb{\xi}|\pmb{\theta})$ in $\pmb{\theta}$, and the interchangeability of the limit and the integral follows by the
dominated convergence theorem since $\ln p(\cdot|\pmb{\theta})$ is dominated by an integrable function. And $\phi(\pmb{\theta})$ is continuous on $\Theta$. Thus we have that for an appropriate $\omega > 0$, the set $\Theta({\mathcal{S}^{N}},\alpha)$ is bounded by $U _\omega = \Theta\backslash V_\omega$ which is an arbitrarily tight neighborhood of $\pmb{\theta}^c$ w.p.1. Therefore, (\ref{2.8}) implies that w.p.1, the distance from any $\pmb{\theta}\in\Theta({\mathcal{S}^{N}},\alpha)$  to $\pmb{\theta}^c$ converges to zero w.r.t. $\mathbb{P}_{\mathcal{S}^N}$. 
\end{remark}

We define the true uncertainty set $\Xi^c$ as the one which satisfies the true chance constraint implication (\ref{1.3}) by setting the uncertainty set of parameters to $\Theta({\mathcal{S}^{N}},\alpha)=\{\pmb{\theta}^c\}$. According to Theorem \ref{thm2.1} and our schema, $\Xi^c$ is a closed and convex set such that $\delta^*(\pmb{v}|\Xi^c)=\pmb{v}^T\pmb{B}(\pmb{\theta}^c,\epsilon)$.

\begin{theorem}\label{thm3.5}
    Suppose that the upper bound $\pmb{B}(\pmb{\theta},\epsilon)$ is continuous in terms of $\pmb{\theta}\in\Theta$, then  under Assumption \ref{ass3.1}, the probability of the event that the uncertainty set $\Xi({\mathcal{S}^{N}},\epsilon,\alpha))$ converges to the true set $\Xi^c$ tends to one with respect to $\mathbb{P}_{\mathcal{S}^{N}}$, written as $\lim_{N\to\infty}\mathbb{P}_{\mathcal{S}^{N}}(\Xi^c = \lim \Xi({\mathcal{S}^{N}},\epsilon,\alpha)))=1$, w.p.1.
\end{theorem}
\begin{proof}
With the continuity of the upper bound $\pmb{B}(\pmb{\theta},\epsilon)$ and the compactness of $\Theta({\mathcal{S}^{N}},\alpha)$, we have that for any $\pmb{\theta}\in\Theta({\mathcal{S}^{N}},\alpha)$, the distance  from $\pmb{B}(\pmb{\theta},\epsilon)$ to $\pmb{B}(\pmb{\theta}^c,\epsilon)$ converges to zero  w.p.1 with respect to $\mathbb{P}_{\mathcal{S}^N}$, which means that
$\lim_{N\to\infty}\mathbb{P}_{\mathcal{S}^{N}}(\delta^*(\pmb{v}|\Xi({\mathcal{S}^{N}},\epsilon,\alpha)\text{ converges to }\delta^*(\pmb{v}|\Xi^c)))=1$.

It is known from Theorem 11 in \cite{refcon} that, if the support functions $\{\delta^*(\pmb{v}|\Xi({\mathcal{S}^{N}},\epsilon,\alpha))\}$ and $\delta^*(\pmb{v}|\Xi^c)$ are equi-lower semicontinuous, then 
\[\lim_{N\to\infty}\delta^*(\pmb{v}|\Xi({\mathcal{S}^{N}},\epsilon,\alpha))=\delta^*(\pmb{v}|\Xi^c)\text{ if and only if } \Xi^c=\lim\Xi({\mathcal{S}^{N}},\epsilon,\alpha).\]
Fortunately, it is easy to prove that the support functions $\{\delta^*(\pmb{v}|\Xi({\mathcal{S}^{N}},\epsilon,\alpha))\}$ and $\delta^*(\pmb{v}|\Xi^c)$ are equi-continuous due to the compactness of $\Theta$. Thus $\lim_{N\to\infty}\mathbb{P}_{\mathcal{S}^{N}}(\Xi^c=\lim\Xi({\mathcal{S}^{N}},\epsilon,\alpha))=1.$
\end{proof}

With the above theoretical foundation, we can now consider the application of our schema Algorithm 1 under different circumstances.	
\section{BDRCC Approximation with discrete distributions}
In this section, we consider the case of  a discrete distribution  $\mathbb{P}_{\pmb{\theta}^c}$ with finite support, i.e., supp$(\mathbb{P}_{\pmb{\theta}^c})=\{r_1,\cdots,r_{n} \}$. Under this situation, it is natural to set $\mathbb{P}_{\pmb{\theta}^c}:=\pmb{\theta}^c$ with the component being $\pmb{\theta}^c_i=\mathbb{P}_{\pmb{\theta}^c}(\tilde{\pmb{\xi}}=r_i)$ for $i =1,\cdots,n$. Let $\Delta_n\equiv\{\pmb{\theta}\in\mathbb{R}^n_+:\pmb{e}^T\pmb{\theta}=1\}$ denote the probability simplex.
	
To start with, we introduce a well-known property of the Conditional Value at Risk (CVaR) (\cite{ref1,ref43}) which is defined as
\[\text{CVaR}_\epsilon^{\mathbb{P}}(\pmb{v}^T\tilde{\pmb{\xi}})\equiv\min_t\left\{t+\frac{1}{\epsilon}\mathbb{E}^{\mathbb{P}}\left[\left(\tilde{\pmb{\xi}}^Tv-t\right)^+\right]\right\},\]
with respect to a fixed measure $ \mathbb{P}$ and vector $\pmb{v}\in\mathbb{R}^d$.
\begin{lemma} (Theorem 2 in \cite{ref43})\label{lem4.1}
    Suppose supp$(\mathbb{P}_{\pmb{\theta}} )=\{r_1,\cdots,r_{n} \}$ and  $\pmb{\theta}_j=\mathbb{P}_{\pmb{\theta}}(\tilde{\pmb{\xi}}=r_j),\ j=1,\cdots,n$. Let
	\begin{equation}
		\Xi^{\text{CVaR}^{\mathbb{P}_{\pmb{\theta}}}_\epsilon}=\left\{\pmb{\xi}\in\mathbb{R}^d:\pmb{\xi}=\sum_{j=1}^{n}q_jr_j,q\in\Delta_n,q\leq\frac{1}{\epsilon}\pmb{\theta} \right\}.\label{3.1}
	\end{equation}
	Then, $\delta^*(\pmb{v}|\Xi^{\text{CVaR}^{\mathbb{P}_{\pmb{\theta}}}_\epsilon})=\text{CVaR}^{\mathbb{P}_{\pmb{\theta}}}(\pmb{v}^T\tilde{\pmb{\xi}})$.
\end{lemma}
Based on Lemma \ref{lem4.1}, the worst-case CVaR over the Bayesian credible interval configured as
	\[\Theta({\mathcal{S}^{N}},\alpha)=\{\pmb{\theta}\in\Delta_n:||[I(\hat{\pmb{\theta}})]^{1/2}  (\pmb{\theta} - \hat{\pmb{\theta}})||_2\leq z_{1-\frac{\alpha}{2}}\}\]
	will lead to an uncertainty set as the following theorem shows.	
\begin{theorem}\label{thm4.1}
     Suppose that $\mathbb{P}_{\pmb{\theta}^c}$ has a known, finite support $\{r_1,\cdots,r_{n} \}$. The following uncertainty set $\Xi({\mathcal{S}^{N}},\epsilon,\alpha)$ satisfies the BDRCC (\ref{BDRCC1}) for all $0<\epsilon<1$:
	\[\Xi({\mathcal{S}^{N}},\epsilon,\alpha)=\left\{\pmb{\xi}\in\mathbb{R}^d:\pmb{\xi}=\sum_{j=1}^{n}q_jr_j,q\in\Delta_n,q\leq\frac{1}{\epsilon}\pmb{\theta} ,\pmb{\theta}\in\Theta({\mathcal{S}^{N}},\alpha)\right\}.\]
	The corresponding support function is given by
			\begin{equation*}
		\begin{split}
	\delta^*(\pmb{v}|\Xi({\mathcal{S}^{N}},\epsilon,\alpha))=&\min_{\beta,w,\eta,\lambda}\beta+\frac{1}{\epsilon}\left(\hat{\pmb{\theta}}^Tw+d^T\eta+z_{1-\frac{\alpha}{2}}^2\lambda-\sum_{j=1}^n\frac{(b_j^Tw-c_j^T\eta)^2}{4\lambda}\right)	\\
	&\text{s.t. }w,\lambda,\eta\geq0,\ r_j^T\pmb{v}-w_j\leq\beta,\ j=1,\cdots,n 	.
		\end{split}
	\end{equation*}
where  $B=[I(\hat{\pmb{\theta}})]^{-1/2}$, $C=\left[\pmb{e}^T[I(\hat{\pmb{\theta}})]^{-1/2},\ 	-\pmb{e}^T[I(\hat{\pmb{\theta}})]^{-1/2},\ 
-[I(\hat{\pmb{\theta}})]^{-1/2}\right]^T$
 and $d=\left[	1-\pmb{e}^T\hat{\pmb{\theta}},\ 
 -1+\pmb{e}^T\hat{\pmb{\theta}},\ 
 -\hat{\pmb{\theta}}\right]^T$.
\end{theorem}
\begin{proof}
    With Lemma \ref{lem4.1},	an upper bound to $\sup_{\pmb{\theta}\in\Theta({\mathcal{S}^{N}},\alpha)}\text{VaR}_\epsilon^{{\mathbb{P}_{\pmb{\theta}}}}(\pmb{v}^T\tilde{\pmb{\xi}})$ can be derived as 
		\begin{equation*}
		\begin{split}
			\sup_{\pmb{\theta}\in\Theta({\mathcal{S}^{N}},\alpha)}\text{VaR}_\epsilon^{{\mathbb{P}_{\pmb{\theta}}}}(\pmb{v}^T\tilde{\pmb{\xi}})\leq&\sup_{\pmb{\theta}\in\Theta({\mathcal{S}^{N}},\alpha)}\text{CVaR}_\epsilon^{{\mathbb{P}_{\pmb{\theta}}}}(\pmb{v}^T\tilde{\pmb{\xi}})\\
			=&\sup_{\pmb{\theta}\in\Theta({\mathcal{S}^{N}},\alpha)}\max_{\pmb{\xi}\in\Xi^{\text{CVaR}^\mathbb{P}_\epsilon}}\pmb{v}^T\pmb{\xi}	\\	
			=&\max_{\pmb{\xi}\in\Xi({\mathcal{S}^{N}},\epsilon,\alpha)}\pmb{v}^T\pmb{\xi}\\
			=&\delta^*(\pmb{v}|\Xi({\mathcal{S}^{N}},\epsilon,\alpha)).
		\end{split}
	\end{equation*}
Here $\Xi({\mathcal{S}^{N}},\epsilon,\alpha)$ is defined as $\{\pmb{\xi}\in\Xi^{\text{CVaR}_\epsilon^{{\mathbb{P}_{\pmb{\theta}}}}},\ \pmb{\theta}\in\Theta({\mathcal{S}^{N}},\alpha)\}$.

As for $\delta^*(\pmb{v}|\Xi({\mathcal{S}^{N}},\epsilon,\alpha))$, we have
	\[\delta^*(\pmb{v}|\Xi({\mathcal{S}^{N}},\epsilon,\alpha))=\inf_{w\geq0}\left\{\max_{q\in\Delta_n}\sum_{j=1}^nq_j(r_j^T\pmb{v}-w_j)+\frac{1}{\epsilon}\max_{\pmb{\theta}\in\Theta({\mathcal{S}^{N}},\alpha)}w^T\pmb{\theta}\right\}\]
according to the Lagrangian duality theory. The optimal value of the first maximization within the brace has an analytic expression $\beta=\max_j\{r_j^T\pmb{v}-w_j\}$. We now concentrate on the second maximization within the brace. Denote $\zeta=[I(\hat{\pmb{\theta}})]^{1/2}  (\pmb{\theta} - \hat{\pmb{\theta}})$, thus $\pmb{\theta}=[I(\hat{\pmb{\theta}})]^{-1/2}\zeta+\hat{\pmb{\theta}}$. The uncertainty set $\Theta({\mathcal{S}^{N}},\alpha)$ can be reformulated as 
	\[\Theta({\mathcal{S}^{N}},\alpha)=\{\zeta:||\zeta||_2\leq z_{1-\frac{\alpha}{2}},[I(\hat{\pmb{\theta}})]^{-1/2}\zeta+\hat{\pmb{\theta}}\in\Delta_n\}\]
	and the second maximization can be rewritten as
	\begin{equation}
		\begin{split}
			\max_{\pmb{\theta}\in\Theta({\mathcal{S}^{N}},\alpha)}w^T\pmb{\theta}=&\max_{\zeta}\left\{w^T\left[[I(\hat{\pmb{\theta}})]^{-1/2}\zeta+\hat{\pmb{\theta}}\right]\bigg|\zeta\in\Theta({\mathcal{S}^{N}},\alpha)\right\}\\
			=&\max_{\zeta}\left\{w^T\left[[I(\hat{\pmb{\theta}})]^{-1/2}\zeta+\hat{\pmb{\theta}}\right]\bigg| ||\zeta||_2\leq z_{1-\frac{\alpha}{2}},[I(\hat{\pmb{\theta}})]^{-1/2}\zeta+\hat{\pmb{\theta}}\in\Delta_n\right\}	\\	
			=&\max_{\zeta}\left\{w^T\left[[I(\hat{\pmb{\theta}})]^{-1/2}\zeta+\hat{\pmb{\theta}}\right]\bigg| \sum_{j=1}^n\zeta_j^2\leq z_{1-\frac{\alpha}{2}}^2,\pmb{e}^T\left[[I(\hat{\pmb{\theta}})]^{-1/2}\zeta+\hat{\pmb{\theta}}\right]=1,[I(\hat{\pmb{\theta}})]^{-1/2}\zeta+\hat{\pmb{\theta}}\geq0\right\}.		
	\end{split}\label{3.2}
	\end{equation}
We denote $B=[I(\hat{\pmb{\theta}})]^{-1/2}$, $C=\left[	\pmb{e}^T[I(\hat{\pmb{\theta}})]^{-1/2},\ 	-\pmb{e}^T[I(\hat{\pmb{\theta}})]^{-1/2},\ 	-[I(\hat{\pmb{\theta}})]^{-1/2}\right]^T$ and $d=\left[	1-\pmb{e}^T\hat{\pmb{\theta}},\ -1+\pmb{e}^T\hat{\pmb{\theta}},\ -\hat{\pmb{\theta}}\right]^T$ for notational convenience.

The Lagrange function for the optimization problem on the right-hand-side of (\ref{3.2}) is given by
	\[L(\zeta,\lambda,\eta)=w^T(\hat{\pmb{\theta}}+B\zeta)+z_{1-\frac{\alpha}{2}}^2\lambda-\lambda\sum_{j=1}^n\zeta_j^2+\eta^T(d-C\zeta)\]
	and the dual objective function is
	\[g(\lambda,\eta)=\max_\zeta L(\zeta,\lambda,\eta).\]
So we have
	\begin{equation*}
		\begin{split}
		g(\lambda,\eta)=&\hat{\pmb{\theta}}^Tw+d^T\eta+z_{1-\frac{\alpha}{2}}^2\lambda+\max_{\zeta}\sum_{j=1}^n\left(\zeta_j(b_j^Tw-c_j^T\eta)-\lambda\zeta_j^2\right)\\
		=&\hat{\pmb{\theta}}^Tw+d^T\eta+z_{1-\frac{\alpha}{2}}^2\lambda+\sum_{j=1}^n\max_{\zeta_j}\left(\zeta_j(b_j^Tw-c_j^T\eta)-\lambda\zeta_j^2\right)\\=&\hat{\pmb{\theta}}^Tw+d^T\eta+z_{1-\frac{\alpha}{2}}^2\lambda+\sum_{j=1}^n\lambda\max_{\zeta_j}\left(\zeta_j(\frac{b_j^Tw-c_j^T\eta}{\lambda})-\zeta_j^2\right)\\	
		=&\hat{\pmb{\theta}}^Tw+d^T\eta+z_{1-\frac{\alpha}{2}}^2\lambda-\sum_{j=1}^n\frac{(b_j^Tw-c_j^T\eta)^2}{4\lambda}.	
		\end{split}
	\end{equation*}
Thus the optimal value of the second maximization problem within the brace of  (\ref{3.2}) is
\[\inf_{\lambda,\eta\geq0}\hat{\pmb{\theta}}^Tw+d^T\eta+z_{1-\frac{\alpha}{2}}^2\lambda-\sum_{j=1}^n\frac{(b_j^Tw-c_j^T\eta)^2}{4\lambda}.	\]
With the above preparations, the support function can be equivalently determined by
			\begin{equation*}
	\begin{split}
	\delta^*(\pmb{v}|\Xi({\mathcal{S}^{N}},\epsilon,\alpha))=&\min_{\beta,w,\eta,\lambda}\beta+\frac{1}{\epsilon}\left(\hat{\pmb{\theta}}^Tw+d^T\eta+z_{1-\frac{\alpha}{2}}^2\lambda-\sum_{j=1}^n\frac{(b_j^Tw-c_j^T\eta)^2}{4\lambda}\right)	\\
		&\text{s.t.}\ w,\lambda,\eta\geq0,\ r_j^T\pmb{v}-w_j\leq\beta,\ j=1,\cdots,n.	.
	\end{split}
\end{equation*}
\end{proof}

\begin{example}\label{example4.1}
We adopt a Dirichlet prior for $\tilde{\pmb{\theta}}$. Recall that $\tilde{\pmb{\theta}}$ follows a Dirichlet distribution with parameter $\tau'$ if it admits the pdf $p(\pmb{\theta})=B(\tau')^{-1}\prod_{j=1}^{n}{\theta}_j^{\tau_j'-1}$, where $\tau_j'>0$ for all $j$ and $B(\tau')$ is a normalizing constant. The Dirichlet distribution is a conjugate prior in this setting, meaning that the posterior distribution is also Dirichlet with updated parameters $\tau$, $\tau_j=\tau_j'+\sum_{i=1}^{N}\mathbb{I}({\pmb{\xi}}^i=r_j),\ j=1,\cdots,n$. When $\tau'=\pmb{e}$, the Dirichlet distribution reduces to a uniform distribution, a common uninformative prior.  Given data set ${\mathcal{S}^N}$, we can determine the posterior mode ${\hat{\pmb{\theta}}}$ by maximizing a posteriori estimation, and calculate the observed information matrix
\[I(\hat{\pmb{\theta}})= 
\begin{pmatrix}
	(\tau_{1}-1)\frac{1}{\hat{\theta}_{1}^2} & 0 & \cdots & 0\\
	0 & (\tau_{2}-1)\frac{1}{\hat{\theta}_{2}^2} & \cdots & 0\\
	\vdots & \vdots & \ddots & \vdots\\
	0 & 0 & \cdots & (\tau_{n}-1)\frac{1}{\hat{\theta}_{n}^2}\\
\end{pmatrix}.
\] 
The observed information $I(\hat{\pmb{\theta}})$ is a diagonal matrix which means that $\pmb{\theta}$ has independent components w.r.t. $\mathbb{P}_{\mathcal{S}^N}$, thus we can get the approximate posterior distribution for each parameter as
\begin{equation}
	p(\theta_{j}|{\mathcal{S}^N})\approx N(\hat{\theta}_{j},(\tau_{j}-1)^{-1}\hat{\theta}_{j}^2),\  j=1,\cdots,n.\label{3.3}
\end{equation}

By setting $\alpha'=1-\sqrt[n]{1-\alpha}$ for any $j=1,\cdots,n$, the multivariate credible interval which rejects if ${\theta}_j$ doesn't belong to the univariate credible interval at level $\alpha'$ is a valid credible interval since
\[\mathbb{P}_{\mathcal{S}^N}\left( \theta_{i}\text{ is accepted at level }\alpha'\text{ for all }j =1,\cdots,n\right)=\prod_{i=1}^n(1-\alpha')=1-\alpha\]
by independence. Therefore, the Bayesian credible interval  has the following (approximate) $(1 - \alpha)$ credible set around  $\mathbb{P}_{\hat{\pmb{\theta}}}$ with respect to $\mathbb{P}_{\mathcal{S}^N}$:
\[\left\{\mathbb{P}_{\pmb{\theta}}|\pmb{\theta}\in\Theta({\mathcal{S}^N},\alpha)\right\}\text{ which satisfies }\mathbb{P}_{\mathcal{S}^N}\left(\mathbb{P}_{\pmb{\theta}^c}\in\{\mathbb{P}_{\pmb{\theta}}|\pmb{\theta}\in\Theta({\mathcal{S}^N},\alpha)\}\right)\geq1-\alpha,\]
with
\begin{equation}
	\Theta({\mathcal{S}^N},\alpha)=\left\{\pmb{\theta}\in\Delta_n:|\theta_{j}-\hat{\theta}_{j}|\leq \frac{z_{1-\frac{\alpha'}{2}}}{(\tau_{j}-1)^{-1/2}\hat{\theta}_{j}},\  j=1,\cdots,n\right\}.\label{3.4}
\end{equation}

Thus (\ref{3.2}) can be reformulated as
\[\max_{\pmb{\theta}\in\Theta({\mathcal{S}^{N}},\alpha)}w^T\pmb{\theta}=\max_{\pmb{\theta}} \left\{w^T\pmb{\theta}\bigg|\sum_{j=1}^n\theta_{j}=1,\  |\theta_{j}-\hat{\theta}_{j}|\leq \frac{z_{1-\frac{\alpha'}{2}}}{(\tau_{j}-1)^{-1/2}\hat{\theta}_{j}},\  j=1,\cdots,n\right\}.\]
Denote \[A=\left[
	\pmb{e},
	-\pmb{e},
	\pmb{e}_1,
-\pmb{e}_1,
\cdots,
	\pmb{e}_n,
-\pmb{e}_n
\right],\] and \[C=\left[1,-1,\frac{z_{1-\frac{\alpha'}{2}}}{(\tau_{1}-1)^{-1/2}\hat{\theta}_{1}}+\hat{\theta}_1,\frac{z_{1-\frac{\alpha'}{2}}}{(\tau_{1}-1)^{-1/2}\hat{\theta}_{1}}-\hat{\theta}_1,\cdots,\frac{z_{1-\frac{\alpha'}{2}}}{(\tau_{n}-1)^{-1/2}\hat{\theta}_{n}}+\hat{\theta}_n,\frac{z_{1-\frac{\alpha'}{2}}}{(\tau_{n}-1)^{-1/2}\hat{\theta}_{n}}-\hat{\theta}_n\right]^T.\]

The dual problem for problem (\ref{3.2}) under the above setting is:
\[\min C^TX\ \text{s.t.}\ AX\leq w,\ X\geq0,\]
and the support function can be determined as the optimal value of the following linear programming problem:
			\begin{equation*}
	\begin{split}
		\delta^*(\pmb{v}|\Xi({\mathcal{S}^{N}},\epsilon,\alpha))=&\min_{\beta,w}\beta+\frac{1}{\epsilon}C^TX	\\
		\text{s.t.}\ &0\leq w,\ r_j^T\pmb{v}-w_j\leq\beta,\ j=1,\cdots,n, \\
		&AX\leq w,\ X\geq0.
	\end{split}
\end{equation*}    
\end{example}
\section{BDRCC Approximation with continuous distributions}
 When the underlying distribution $\mathbb{P}_{\pmb{\theta}^c}$ is a high-dimensional distribution with continuous support, most of the existent approaches, such as \cite{ref8,ref19_2,ref44}, are only applicable for distributions with independent components or particular dependence structures. A significant advantage of our schema is that it can cope with various dependence forms, which will be illustrated in this section. Without loss of generality, we assume that one can observe $N$ samples of each marginal distribution.
 \subsection{Independent Marginal Distributions}
Suppose that $\mathbb{P}_{\pmb{\theta}^c}$ has a continuous support with independent marginal distributions $\mathbb{P}_{\pmb{\theta}_i^c}$, $i=1,\cdots,d.$ That is, the components $\tilde{\xi}_{1},\cdots,\tilde{\xi}_{d}$ are independent. In this case, we can combine univariate credible intervals for individual marginal distributions to build a multivariate credible interval.
 
Let $\alpha'=1-\sqrt[d]{1-\alpha}$, the credible interval for the $i$-th marginal distribution at level $\alpha'$ is
 \[\Theta_{i}({\mathcal{S}^{N}},\alpha')=\{\pmb{\theta}_i:||[I(\hat{\pmb{\theta}}_i)]^{1/2}  (\pmb{\theta}_i - \hat{\pmb{\theta}}_i)||_2\leq z_{1-\frac{\alpha'}{2}}\}.\] 
Then we have $\mathbb{P}_{\mathcal{S}^N} (\mathbb{P}_{\pmb{\theta}_i}$ is accepted at level $\alpha'$ for $i =1,\cdots,d$)	=$\prod_{i=1}^d(1-\alpha')=1-\alpha$ by independence. The corresponding multivariate credible interval  is
 \[\Theta({\mathcal{S}^{N}},\alpha)=\{\pmb{\theta}:\pmb{\theta}=\otimes_{i=1}^d\pmb{\theta}_i,\pmb{\theta}_i\in\Theta_{i}({\mathcal{S}^{N}},\alpha')\},\]
 and thus the ambiguity set $\mathcal{P}_{\Theta({\mathcal{S}^{N}},\alpha)}=\left\{\mathbb{P}_{\pmb{\theta}}:\mathbb{P}_{\pmb{\theta}}=\prod_{i=1}^d\mathbb{P}_{\pmb{\theta}_i}\ \pmb{\theta}_i\in\Theta_{i}({\mathcal{S}^{N}},\alpha')\right\}$  will contain the true distribution with probability at least $1-\alpha$ w.r.t. $\mathbb{P}_{\mathcal{S}^N}$.

 \begin{theorem}\label{thm5.1}
      Suppose that $\mathbb{P}_{\pmb{\theta}^c}$ has independent marginal distributions $\mathbb{P}_{\pmb{\theta}_1^c},\cdots,\mathbb{P}_{\pmb{\theta}_d^c}$. Then for any $0<\epsilon<1$, we have that,  the following $\Xi^I({\mathcal{S}^{N}},\epsilon,\alpha)$ implies the BDRCC (\ref{BDRCC1}),
\[\Xi^I({\mathcal{S}^{N}},\epsilon,\alpha)=\left\{\pmb{\xi}\in\mathbb{R}^d:\xi_i=\text{VaR}_{1-\sqrt[d]{1-\epsilon}}^{\mathbb{P}_{\pmb{\theta}_i}}(\tilde{\xi}_i),\ \pmb{\theta}_i\in\Theta_{i}({\mathcal{S}^{N}},\alpha'),\ i=1,\cdots,d\right\}.\]
 Moreover,
 \[\delta^*(\pmb{v}|\Xi^I({\mathcal{S}^{N}},\epsilon,\alpha))=\max_{\pmb{\theta}\in\Theta({\mathcal{S}^{N}},\alpha)}\sum_{i=1}^dv_i\text{VaR}_{1-\sqrt[d]{1-\epsilon}}^{\mathbb{P}_{\pmb{\theta}_i}}(\tilde{\xi}_i).\]
 \end{theorem}
 \begin{proof}
      When the marginals of $\mathbb{P}_{\pmb{\theta}}$ are known and independent, it is known from Theorem 4.1 in \cite{ref27} that
\[\text{VaR}_\epsilon^{\mathbb{P}_{\pmb{\theta}}}(\sum_{i=1}^dv_i\tilde{\xi}_i)\leq\min_{\lambda:\prod_{i=1}^d\lambda_i=1-\epsilon}\sum_{i=1}^dv_i\text{VaR}_{1-\lambda_i}^{\mathbb{P}_{\theta_i}}(\tilde{\xi}_i)\leq\sum_{i=1}^dv_i\text{VaR}_{1-\sqrt[d]{1-\epsilon}}^{\mathbb{P}_{\theta_i}}(\tilde{\xi}_i),\]
where the last inequality is obtained by letting $\lambda_i={\sqrt[d]{1-\epsilon}}$ for all $i$.

The above upper bound implies the following worst-case bound
\[\sup_{\pmb{\theta}\in\Theta({\mathcal{S}^{N}},\alpha)}\text{VaR}_\epsilon^{\mathbb{P}_{\pmb{\theta}}}(\sum_{i=1}^dv_i\tilde{\xi}_i)\leq\sup_{\pmb{\theta}_i\in\Theta_{i}({\mathcal{S}^{N}},\alpha')}\sum_{i=1}^dv_i\text{VaR}_{1-\sqrt[d]{1-\epsilon}}^{\mathbb{P}_{\theta_i}}(\tilde{\xi}_i).\]
 With Theorem \ref{thm3.1}, we can immediately deduce that $\delta^*(\pmb{v}|\Xi^I({\mathcal{S}^{N}},\epsilon,\alpha))$ is the support function of $\Xi({\mathcal{S}^{N}},\epsilon,\alpha)$.
 \end{proof}

\begin{example}(Ambiguity set for normal marginal distributions)
    We consider $d$ independent random variables with $\mathbb{P}_{\pmb{\theta}^c_i}=N(\mu_i^c ,(\sigma_i^c) ^2)$, $\mu_i\in\mathbb{R}$, ${\sigma_i^c}>0$, for $1\leq i \leq d$. Recall that for a random variable $\tilde{\xi}_i$ following a normal distribution $\mathbb{P}_{\pmb{\theta}_i}=N(\mu_i ,\sigma_i^2)$ where $\pmb{\theta}_i=(\mu_i,\sigma_i)$, we have
\[\text{VaR}_\epsilon^{\mathbb{P}_{\pmb{\theta}_i}}(\tilde{\xi}_i)=\mu_i +\sigma_iz_{1-\epsilon},\]
where $z_{1-\epsilon}$ denotes the $1-\epsilon$ quantile of the standard normal  distribution. Thus we can get that
\[\sup_{\pmb{\theta}\in\Theta({\mathcal{S}^{N}},\alpha)}\text{VaR}_\epsilon^{\mathbb{P}_{\pmb{\theta}}}(\sum_{i=1}^dv_i\tilde{\xi}_i)\leq\sum_{i=1}^d\sup_{\pmb{\theta}_i\in\Theta_{i}({\mathcal{S}^{N}},\alpha')}v_i\left(\mu_i +\sigma_iz_{\sqrt[d]{1-\epsilon}}\right).\]
\end{example}

From which we can deduce that
\[\Xi^I({\mathcal{S}^{N}},\epsilon,\alpha)=\left\{\pmb{\xi}\in\mathbb{R}^d:\xi_i=\mu_i +\sigma_i z_{\sqrt[d]{1-\epsilon}},\ (\mu_i,\sigma_i)\in\Theta_{i}({\mathcal{S}^{N}},\alpha'),\ i=1,\cdots,d\right\}\]
satisfies the BDRCC (\ref{BDRCC1}) and the support function of $\Xi^I({\mathcal{S}^{N}},\epsilon,\alpha)$ is
 \[\delta^*(\pmb{v}|\Xi^I({\mathcal{S}^{N}},\epsilon,\alpha))=\max_{\pmb{\theta}\in\Theta({\mathcal{S}^{N}},\alpha)}\sum_{i=1}^dv_i\mu_i +v_i\sigma_i z_{\sqrt[d]{1-\epsilon}}.\]
 \subsection{Copula-based Dependent Marginal Distributions}
 In probability theory, a copula is a function that describes the dependence among random variables. It allows us to model a multi-dimensional random vector with any marginal distributions. Copulas are often used to describe the joint distribution of variables in multivariate analysis, particularly in finance \cite{copulasfinance}, insurance \cite{copulasinsurance}, and risk management \cite{copulasrisk}.
  
By specifying a copula, one can model the correlation between variables without restricting specific forms about the marginal distributions. Nevertheless, we have not seen any study which explores the DRCC problem under the copula function framework.
 
It has been shown in many studies like \cite{ref2,3,pda1} that copula function can improve the worst-case VaR bounds. Inspired by this, we characterize correlations through copula functions and thus construct more accurate uncertainty sets.  Throughout this subsection, we extensively use the concept of a copula to model the dependence structure of the marginal distributions $\mathbb{P}_{\pmb{\theta}^c_i}$, $1\leq i\leq d$.
 
As we know, a copula $C$ is a $d$-dimensional distribution function on $[0,1]^d$ with uniform marginals. Given a copula $C$ and $d$ univariate distributions $F_{\theta_1} ,\cdots, F_{\theta_d}$, one can define a joint distribution $F_{\pmb{\theta}}$ on $ \mathbb{R}^d$ with these $d$ marginal distributions by	 
 \[F_{\pmb{\theta}}(\xi_1,\cdots,\xi_d)=C(F_{\pmb{\theta}_1}(\xi_1),\cdots,F_{\pmb{\theta}_d}(\xi_d)).\]
 Here $F_{\pmb{\theta}_i}$ is the cumulative distribution function of $\xi_i$ w.r.t. $\mathbb{P}_{\pmb{\theta}_i}$
 \[F_{\pmb{\theta}_i}(\xi_i)=\mathbb{P}_{\pmb{\theta}_i}(-\infty,\xi_i],\]
 and $F_{\pmb{\theta}}$ is the cumulative distribution function of $\pmb{\xi}$ w.r.t. $\mathbb{P}_{\pmb{\theta}}$
 \[F_{\pmb{\theta}}(\xi_1,\cdots,\xi_d)=\mathbb{P}_{\pmb{\theta}}(-\infty,\xi_1]\times\cdots\times(-\infty,\xi_d].\]	 
 We denote the above process as
$\mathbb{P}_{\pmb{\theta}}=C(\mathbb{P}_{\pmb{\theta}_1},\mathbb{P}_{\pmb{\theta}_2},\cdots,\mathbb{P}_{\pmb{\theta}_i},\cdots,\mathbb{P}_{\pmb{\theta}_d})$ for short.
 
From the results in subsection 5.1, the credible interval for the $i$-th parameter can be defined as
 \[\Theta_{i}({\mathcal{S}^{N}},\alpha/d)=\{\pmb{\theta}_i:||[I(\hat{\pmb{\theta}}_i)]^{1/2}  (\pmb{\theta}_i - \hat{\pmb{\theta}}_i)||_2\leq z_{1-\frac{\alpha}{2d}}\},\]
 and the multivariate credible interval can be then defined as 
 \[\Theta({\mathcal{S}^{N}},\alpha)=\{\pmb{\theta}:\pmb{\theta}=\otimes_{i=1}^d\pmb{\theta}_i,\pmb{\theta}_i\in\Theta_{i}({\mathcal{S}^{N}},\alpha/d),\ 1\leq i\leq d\}.\]
Then with the copula $C$,
  \[\mathcal{P}_{\Theta({\mathcal{S}^{N}},\alpha)}=\left\{\mathbb{P}_{\pmb{\theta}}:\mathbb{P}_{\pmb{\theta}}=C(\mathbb{P}_{\pmb{\theta}_1},\mathbb{P}_{\pmb{\theta}_2},\cdots,\mathbb{P}_{\pmb{\theta}_i},\cdots,\mathbb{P}_{\pmb{\theta}_d}),\ \pmb{\theta}_i\in\Theta_{i}({\mathcal{S}^{N}},\alpha/d),\ 1\leq i\leq d\right\}\]
  will contain the true probability distribution with at least $1-\alpha$ w.r.t. $\mathbb{P}_{\mathcal{S}^N}$.
   
 To demonstrate the concrete application of the above result for constructing the uncertainty set, we first introduce an important concept in copula theory. 

\begin{definition}
    Positive dependence is that a copula $C_\Xi$ of the random variable $\pmb{\xi}$ satisfies $C_\Xi(u) \geq C_l(u)$, $u \in [0, 1]^d$. Here $C_l : [0, 1]^d\to [0, 1]$ is a componentwise increasing function satisfying
 \[W(u)\leq C_l(u)\leq M(u),\ u\in[0,1]^d\]
 with $W(u)=(u_1+\cdots+u_d-n+1)_+$ and $M(u)=\min\{u_1,\cdots,u_d\}$.  
\end{definition}
 
 As an important type of copulas, positive dependence is particularly relevant in the context of BDRCC, since it can have significant implications for the joint behavior of random variables. For example, to get an improved bound for the VaR estimate, the positive dependence assumption has been adopted for a long time (e.g., \cite{ref27,pda1}). Two important results about positive dependence in \cite{main3} are shown as follows.

 \begin{lemma}\label{lem5.1}
     For the random vector $(\tilde{\xi}_1,\cdots,\tilde{\xi}_d)$ with fixed marginal distribution functions $F_1,\cdots,F_d$, if $C_\Xi\geq C_l$, then
 \[\text{VaR}_{\epsilon}(\sum_{i=1}^d\tilde{\xi}_i)\leq\sum_{i=1}^dF_i^{-1}(\epsilon^*),\]
 where $\epsilon^*=\delta_{C_l}^{-1}(\epsilon)$ and $\delta_{C_l}: [0, 1] \to[0, 1]$ is defined as $\delta_{C_l}(u)=C_l(u,\cdots,u)$.
 \end{lemma}

\begin{lemma}\label{lem5.2}
    Let  $S\subset[0,1]^d$ and $C_l$ be a componentwise increasing function on $[0,1]^d$ such that $W(u)\leq C_l(u)\leq M(u) $, $u\in[0,1]^d$. Define the bound $B^{S,C_l}:[0,1]^d\to[0,1]$ as
 \[B^{S,C_l}(u)=\max\left(W(u),\sup_{a\in S}\left\{C_l(a)-\sum_{i=1}^d(a_i-u_i)_+\right\}\right).\]
 Then if $C_\Xi(u)\geq C_l(u)$ for all $u\in S$, we have $C_\Xi(u)\geq B^{S,C_l}(u)$ for all $u\in[0,1]^d$.
\end{lemma}
 \subsubsection{Positive dependence in tails}
 Positive dependence in tails refers to a situation where extreme values of one variable tend to be associated with extreme values of other variables. This phenomenon occurs in many areas such as insurance and finance. For example, in the insurance industry, positive tail dependence can occur between different types of insurance claims. If a major hurricane causes widespread property damage, it may also lead to an increase in auto insurance claims due to more accidents on the roads, resulting in positive tail dependence between these two types of claims.
 
 To approximate the resulting BDRCC in this case, we first consider the computation of an upper bound on $\text{VaR}_\epsilon^{{\mathbb{P}_{\pmb{\theta}}}}(\pmb{v}^T\tilde{\pmb{\xi}})$ when positive dependence is assumed only in the upper tails of distributions. 

\begin{theorem}\label{thm5.2}
    Suppose the copula $C_\Xi$ of $(\tilde{\xi}_1,\cdots,\tilde{\xi}_d)$ satisfies that $C_\Xi\geq C_l$ on $S=[\beta,1]^d$, with $\epsilon\leq1-\beta$, then 
 \[\Xi^C({\mathcal{S}^{N}},\epsilon,\alpha)=\left\{\pmb{\xi}\in\mathbb{R}^d:\xi_i=\text{VaR}_{1-\delta_{C_l}^{-1}(1-\epsilon)}^{\mathbb{P}_{\pmb{\theta}_i}}(\tilde{\xi}_i),\ \pmb{\theta}_i\in\Theta_{i}({\mathcal{S}^{N}},\alpha/d),\ i=1,\cdots,d\right\}\]
 implies the BDRCC (\ref{BDRCC1})  for all $0<\epsilon<1$.
 Moreover,
 \[\delta^*(\pmb{v}|\Xi^C({\mathcal{S}^{N}},\epsilon,\alpha))=\sum_{i=1}^d\max_{\pmb{\theta}_i\in\Theta_{i}({\mathcal{S}^{N}},\alpha/d)}v_i\text{VaR}_{1-\delta_{C_l}^{-1}(1-\epsilon)}^{\mathbb{P}_{\pmb{\theta}_i}}(\tilde{\xi}_i).\]
\end{theorem} 
\begin{proof}
     For a fixed $u\in[1-\epsilon,1]^d$, we have
 \begin{equation}
 	\max_{a\in[\beta,1]^d}\left\{C_l(a)-\sum_{i=1}^d(a_i-u_i)_+\right\}=\max_{a\in[u,1]^d}\left\{C_l(a)-\sum_{i=1}^d(a_i-u_i)\right\}= C_l(u).\label{4.1}
 \end{equation}
 The first equality is deduced by the componentwise increasingness of the copula $C_l$, which guarantees that any change of $a_j < u_j$ to $u_j$ will never affect the maximization in (\ref{4.1}). The second equality holds by Theorem 1.5.1 in \cite{ref9}, which shows the Lipschitz continuity of $C_l$:
 \[C_l(a)-C_l(u)\leq\sum_{i=1}^d(a_i-u_i).\]
 With the definition of $B^{S,C_l}$, we have by (\ref{4.1}) that $C_l=B^{S,C_l}$ for $u\geq1-\epsilon$. Furthermore, we can ensure that $\delta_{B^{S,C_l}}(u)=\delta_{C_l}(u)$. Thus we can derive an upper bound of VaR with the help of Lemma \ref{lem5.1}:
 \begin{equation}
 	\text{VaR}_\epsilon^{\mathbb{P}_{\pmb{\theta}}}(\pmb{v}^T\tilde{\pmb{\xi}})\leq
 	\sum_{i=1}^dv_iF_{\pmb{\theta}_i}^{-1}(\epsilon^\star)=\sum_{i=1}^dv_i\text{VaR}_{\epsilon^\star}^{\mathbb{P}_{\pmb{\theta}_i}}(\tilde{\xi}_i)\label{4.2}
 \end{equation}
 where $\epsilon^\star=1-\delta_{C_l}^{-1}(1-\epsilon)$.
 
Here, (\ref{4.2}) implies that
 \[\sup_{\pmb{\theta}\in\Theta({\mathcal{S}^{N}},\alpha)}\text{VaR}_\epsilon^{\mathbb{P}_{\pmb{\theta}}}(\pmb{v}^T\tilde{\pmb{\xi}})\leq\sum_{i=1}^d\sup_{\pmb{\theta}_i\in\Theta_{i}({\mathcal{S}^{N}},\alpha/d)}v_i\text{VaR}_{1-\delta_{C_l}^{-1}(1-\epsilon)}^{\mathbb{P}_{\pmb{\theta}_i}}(\tilde{\xi}_i).\]
 With Theorem \ref{thm3.1}, we can see immediately that $\delta^*(\pmb{v}|\Xi^C({\mathcal{S}^{N}},\epsilon,\alpha))$ is the support function of $\Xi({\mathcal{S}^{N}},\epsilon,\alpha)$.
\end{proof}

 \begin{remark}
      Recall the most basic copula function $\Pi(u_1,\cdots,u_d)=\prod_{i=1}^du_i$ means that the marginal distributions are independent. When we set $C_l$ as $\Pi$ in Theorem \ref{lem5.2}, we can get $\delta_{\Pi}(u)=u^d$ and therefore $\delta_{\Pi}^{-1}(1-\epsilon)=\sqrt[d]{1-\epsilon}$. Thus it is easy to see that the uncertainty set $\Xi({\mathcal{S}^{N}},\epsilon,\alpha)$ generated with Theorem \ref{lem5.2} is an upper bound compared with that generated with Theorem \ref{thm5.1} for the reason that $\Theta_{i}({\mathcal{S}^{N}},1-\sqrt[d]{1-\alpha})\subsetneq\Theta_{i}({\mathcal{S}^{N}},\alpha/d)$ for all $1\leq i\leq d$.
 \end{remark}
 \begin{remark}
      Theorem \ref{lem5.2} would give a same VaR bound as Lemma \ref{lem5.1} if the positive dependence assumption $C_\Xi\geq C_l$ only holds in the upper $\epsilon$-tails of the distribution of $\tilde{\pmb{\xi}}$. This means that we can obtain the same uncertainty set $\Xi({\mathcal{S}^{N}},\epsilon,\alpha)$ which implies the BDRCC (\ref{BDRCC1}) under a weaker assumption.
 \end{remark}

 \begin{example}
     In practice, the random vector $\tilde{\pmb{\xi}}$ can often be divided into some subvectors which are independent, while the components of each subvector may be dependent. This setup can be quite general and realistic when the portfolio consists of different risky assets from several industries, as \cite{main2} suggests. The independence among subvectors can be tested statistically to simplify the assumption of correlation among marginal distributions (i.e., the copula of the full distribution). To characterize this situation, we assume that the components $(\tilde{\xi}_1,\cdots,\tilde{\xi}_d)$ of a $d$-dimensional random vector are split into $K$ subvectors. Each subvector $I_k$ has  $d_k=|I_k|$ components and $\sum_{k=1}^Kd_k = d$. Then according to our framework, the copula $C^k_\Xi$ of the subvector $(\tilde{\xi}_{i},i \in I_k)$ satisfies $C^k_\Xi \geq C^k_l$ on $S_k = [1-\epsilon, 1]^{d_k}$, thus $C_\Xi \geq C_l$ on $S = [1-\epsilon, 1]^d$ with
 \[C_l(u_1,\cdots,u_d)=\prod_{k=1}^KC_l^k(u_i,i\in I_k).\]
 Applying Theorem \ref{thm5.2} with the above copula function $C_l$, we can guarantee that
 \[\Xi^C({\mathcal{S}^{N}},\epsilon,\alpha)=\left\{\pmb{\xi}\in\mathbb{R}^d:\xi_i=\text{VaR}_{1-\delta_{C_l}^{-1}(1-\epsilon)}^{\mathbb{P}_{\pmb{\theta}_i}}(\tilde{\xi}_i),\ \pmb{\theta}_i\in\Theta_{i}({\mathcal{S}^{N}},\frac{1-\sqrt[k]{1-\alpha}}{d_k}),\ i=1,\cdots,d\right\}\]
 implies the BDRCC for all $0<\epsilon<1$.
 Moreover,
 \[\delta^*(\pmb{v}|\Xi({\mathcal{S}^{N}},\epsilon,\alpha))=\sum_{i=1}^d\max_{\pmb{\theta}_i\in\Theta_{i}({\mathcal{S}^{N}},\frac{1-\sqrt[k]{1-\alpha}}{d_k})}v_i\text{VaR}_{1-\delta_{C_l}^{-1}(1-\epsilon)}^{\mathbb{P}_{\pmb{\theta}_i}}(\tilde{\xi}_i),\]
 where $C_l(u_1,\cdots,u_d)=\prod_{k=1}^KC_l^i(u_j,j\in I_k)$.
 \end{example}
 \subsubsection{Known central domain}
It has been observed in some literature (\cite{ref2,ref5}) that the joint distribution can be observed in its central domain with statistical methods. Motivated by these applications, we make a similar assumption that a lower bound of the copula function $C_\Xi$ is known in its central domain rather than the upper $(1-\epsilon)$-tails, i.e., $S \cap [1-\epsilon, 1]^d = \emptyset$  in order to complement the case without the tail dependence. Under this situation, we have

\begin{theorem}
    Suppose the copula $C_\Xi$ of $(\tilde{\xi}_1,\cdots,\tilde{\xi}_d)$ satisfies that $C_\Xi\geq C_l$ on $S\subset\mathbb{R}^d$, with $S \cap [1-\epsilon, 1]^d=\emptyset$. Then
 \[\Xi^C({\mathcal{S}^{N}},\epsilon,\alpha)=\left\{\pmb{\xi}\in\mathbb{R}^d:\xi_i=\text{VaR}_{1-\delta_{W}^{-1}(1-\epsilon)}^{\mathbb{P}_{\pmb{\theta}_i}}(\tilde{\xi}_i),\ \pmb{\theta}_i\in\Theta_{i}({\mathcal{S}^{N}},\alpha/d),\ 1\leq i\leq d\right\}\]
 implies the BDRCC (\ref{BDRCC1}) for all $0<\epsilon<1$.
 Moreover,
 \[\delta^*(\pmb{v}|\Xi^C({\mathcal{S}^{N}},\epsilon,\alpha))=\sum_{i=1}^d\max_{\pmb{\theta}_i\in\Theta_{i}({\mathcal{S}^{N}},\alpha/d)}v_i\text{VaR}_{1-\delta_{W}^{-1}(1-\epsilon)}^{\mathbb{P}_{\pmb{\theta}_i}}(\tilde{\xi}_i).\]
\end{theorem}
\begin{proof}
    Since $S \cap [1-\epsilon, 1]^d=\emptyset$, there exists at least one $i$ such that  $u_i <1-\epsilon$ for all $u \in S$. Then $C_l(u) \leq M(u) =\min_{1\leq i\leq d}{u_i}<1-\epsilon$ and we have
 \[\sup_{a\in S}\left\{C_l(a)-\sum_{i=1}^d(a_i-u)_+\right\}\leq\sup_{a\in S}\{C_l(a)\}<1-\epsilon.\]
 It is known from Theorem 3.4 in \cite{main3} that the inequality $\delta_{B^{S,C_l}}(u)\geq1-\epsilon$ can only be satisfied at those values of $u $ such that 
 \[\delta_{B^{S,C_l}}(u)=(nu-n+1)_+=\delta_{W}(u).\]
 Applying Lemma \ref{lem5.1} with $C_l =B^{S,C_l}$ gives us that
 \begin{equation}
 	\text{VaR}_\epsilon^{\mathbb{P}_{\pmb{\theta}}}(\pmb{v}^T\tilde{\pmb{\xi}})\leq\sum_{i=1}^dv_iF_{\pmb{\theta}_i}^{-1}(\epsilon^\star)=\sum_{i=1}^dv_i\text{VaR}_{\epsilon^\star}^{\mathbb{P}_{\pmb{\theta}_i}}(\tilde{\xi}_i),\label{4.3}
 \end{equation}
 where $\epsilon^*=1-\delta_{W}(1-\epsilon)$.
 
We can deduce from (\ref{4.3}) that
 \[\sup_{\pmb{\theta}\in\Theta({\mathcal{S}^{N}},\alpha)}\text{VaR}_\epsilon^{\mathbb{P}_{\pmb{\theta}}}(\pmb{v}^T\tilde{\pmb{\xi}})\leq\sum_{i=1}^d\sup_{\pmb{\theta}_i\in\Theta_{i}({\mathcal{S}^{N}},\alpha/d)}v_i\text{VaR}_{1-\delta_{W}^{-1}(1-\epsilon)}^{\mathbb{P}_{\pmb{\theta}_i}}(\tilde{\xi}_i).\]
 
 With Theorem \ref{thm3.1}, we can ensure that $\delta^*(\pmb{v}|\Xi^C({\mathcal{S}^{N}},\epsilon,\alpha))$ is the support function of $\Xi({\mathcal{S}^{N}},\epsilon,\alpha)$ immediately.
\end{proof}
\subsection{General Marginal Distributions}
In this part, we do not make any assumption about the dependence among the components. This general case can reflect the complex circumstances like that we cannot learn any correlation information about marginal distributions $\mathbb{P}_{\pmb{\theta}_i}$ from the data. This situation occurs in the case of missing data as \cite{main} points out.

\begin{theorem}\label{thm5.4}
     For the random vector $(\tilde{\xi}_1,\cdots,\tilde{\xi}_d)$ with the cumulative distribution $\mathbb{P}_{\pmb{\theta}}$ and marginal distributions $\mathbb{P}_{\pmb{\theta}_i}$, $1\leq i\leq d$. Then
\[\Xi^M({\mathcal{S}^{N}},\epsilon,\alpha)=\left\{\pmb{\xi}\in\mathbb{R}^d:\xi_i=\text{CVaR}_\epsilon^{\mathbb{P}_{\pmb{\theta}_i}}(\tilde{\xi}_i),\ \pmb{\theta}_i\in\Theta_{i}({\mathcal{S}^{N}},\alpha/d),\ i=1,\cdots,d\right\}\]
implies the BDRCC (\ref{BDRCC1}) for all $0<\epsilon<1$.
Moreover,
\[\delta^*(\pmb{v}|\Xi^M({\mathcal{S}^{N}},\epsilon,\alpha))=\max_{\pmb{\theta}\in\Theta({\mathcal{S}^{N}},\alpha)}\sum_{i=1}^dv_i\text{CVaR}_\epsilon^{\mathbb{P}_{\pmb{\theta}_i}}(\tilde{\xi}_i).\]
\end{theorem}
\begin{proof}
    It is known from Theorem 2.1 in \cite{main2} that
\[\text{VaR}_\epsilon^{\mathbb{P}_{\pmb{\theta}}}(\sum_{i=1}^dv_i\tilde{\xi}_i)\leq \text{CVaR}_\epsilon^{\mathbb{P}_{\pmb{\theta}}}(\sum_{i=1}^dv_iY_i),\ Y_i=F_{\pmb{\theta}_i}^{-1}(U),\]
where $U$ is an $U(0, 1)$-distributed random variable. This bound implies the following worst-case upper bound
\[\sup_{\pmb{\theta}\in\Theta({\mathcal{S}^{N}},\alpha)}\text{VaR}_\epsilon^{\mathbb{P}_{\pmb{\theta}}}(\sum_{i=1}^dv_i\tilde{\xi}_i)\leq\sup_{\pmb{\theta}\in\Theta({\mathcal{S}^{N}},\alpha)}\text{CVaR}_\epsilon^{\mathbb{P}_{\pmb{\theta}}}(\sum_{i=1}^dv_iY_i).\]
From which we can get 
\[\sup_{\pmb{\theta}\in\Theta({\mathcal{S}^{N}},\alpha)}\text{VaR}_\epsilon^{\mathbb{P}_{\pmb{\theta}}}(\sum_{i=1}^dv_i\tilde{\xi}_i)\leq\sum_{i=1}^d\sup_{\pmb{\theta}_i\in\Theta_{i}({\mathcal{S}^{N}},\alpha/d)}v_i\text{CVaR}_\epsilon^{\mathbb{P}_{\pmb{\theta}_i}}(Y_i).\]
 With Theorem \ref{thm3.1}, we can ensure that $\delta^*(\pmb{v}|\Xi^M({\mathcal{S}^{N}},\epsilon,\alpha))$ is the support function of $\Xi({\mathcal{S}^{N}},\epsilon,\alpha)$.
\end{proof}
 
 As an application of Theorem \ref{thm5.4}, we consider the following example. 
 
\begin{example}(Ambiguity set for Gamma marginal distributions)
    Suppose that $\mathbb{P}_{\pmb{\theta}^c}$ has components following the distribution $G_{a_i^c,s_i^c}= Gamma(a^c_i, s^c_i)$, with the density function being
\[g_{a^c_i, s^c_i}=\frac{1}{{s^c_i}^{a^c_i}\Gamma(a^c_i)}\xi_i^{a^c_i-1}e^{-\frac{\xi_i}{s^c_i}},\ \ \xi_i\geq0.\]
For a random variable $\tilde{\xi}_i$ following the distribution $\mathbb{P}_{\pmb{\theta}_i}= G_{a_i,s_i}$ where $\pmb{\theta}_{i}=(a_i,s_i)$, it is elementary to show that 
\[\text{CVaR}_\epsilon^{{\mathbb{P}_{\pmb{\theta}_i}}}(\tilde{\xi}_i)=\frac{s_i}{1-\epsilon}\frac{\Gamma(a_i+1)}{\Gamma(a_i)}\left(\bar{G}_{a_i+1,s_i}\left(G_{a_i,s_i}^{-1}(\epsilon)\right)\right)\]
with $\bar{G}_{a_i,s_i}=1-G_{a_i,s_i}$.

Hence, we obtain that 
 \begin{equation}
\sup_{\pmb{\theta}\in\Theta({\mathcal{S}^{N}},\alpha)}\text{VaR}_\epsilon^{\mathbb{P}_{\pmb{\theta}}}(\sum_{i=1}^dv_i\tilde{\xi}_i)\leq\sum_{i=1}^d\sup_{\pmb{\theta}_i\in\Theta_{i}({\mathcal{S}^{N}},\alpha/d)}v_i\left\{\frac{s_i}{1-\epsilon}\frac{\Gamma(a_i+1)}{\Gamma(a_i)}\left(\bar{G}_{a_i+1,s_i}\left(G_{a_i,s_i}^{-1}(\epsilon)\right)\right)\right\}\label{4.4}
\end{equation}
and
\[\Theta_{i}({\mathcal{S}^{N}},\alpha/d)=\{\pmb{\theta}_i:||[I(\hat{\pmb{\theta}}_i)]^{1/2}  (\pmb{\theta}_i - \hat{\pmb{\theta}}_i)||_2\leq z_{1-\frac{\alpha}{2d}}\},\]
where  $\hat{\pmb{\theta}}_i=(\hat{a}_i,\hat{s})$, $\hat{a}_i$ and $\hat{s}$ are  the posterior modes of $a_i$ and $s$, respectively.

Based on (\ref{4.4}), we can further obtain that
\begin{equation*}
    \begin{split}
        \Xi^M({\mathcal{S}^{N}},\epsilon,\alpha)=\bigg\{\pmb{\xi}\in\mathbb{R}^d:\xi_i=\frac{s_i}{1-\epsilon}\frac{\Gamma(a_i+1)}{\Gamma(a_i)}\left(\bar{G}_{a_i+1,s_i}\left(G_{a_i,s_i}^{-1}(\epsilon)\right)\right),\\
        (a_i,s_i)\in\Theta_{i}({\mathcal{S}^{N}},\alpha/d),\ 1\leq i\leq d\bigg\}
    \end{split}
\end{equation*}
implies the BDRCC (\ref{BDRCC1}) for all $0<\epsilon<1$.
\end{example}

\section{Applications}
Notice that in our framework, the uncertainty set $\Xi({\mathcal{S}^{N}},\epsilon,\alpha)$ will always imply the BDRCC (\ref{BDRCC1}) for all $0<\epsilon<1$ because the proposed credible interval $\Theta({\mathcal{S}^{N}},\alpha)$ does not depend on $\epsilon$. This allows us to obtain better-performing uncertainty sets without making more assumptions than the existing literature. To further illustrate the advantages of our framework, we assess the performance of the distributionally robust optimization problem derived from our BDRCC on a portfolio selection problem and a queuing problem. 
\subsection{Portfolio Selection}
Portfolio management is an extensively studied class of problems in robust optimization. One can refer to \cite{review}  for a comprehensive review. From a chance constraint perspective (\cite{drccps}), a typical portfolio selection problem with our proposed BDRCC can be written as:
\begin{equation}
	\max_{\pmb{x},t}\left\{t:\inf_{\pmb{\theta}\in\Theta_{\mathcal{S}^N}}\mathbb{P}_{\pmb{\theta}}(\tilde{\pmb{\xi}}^T\pmb{x}\geq t)\geq1-\epsilon,\ \pmb{e}^T\pmb{x}=1,\ \pmb{x}\geq0\right\}.\label{new5.1}
\end{equation}
Here $\pmb{x}\in\mathbb{R}^d$ is the proportional portfolio vector and  $\tilde{\pmb{\xi}}\in\mathbb{R}^d$ is the random return vector of $d$ risky assets with the underlying distribution $\mathbb{P}_{\pmb{\theta}^c}$. The goal of problem (\ref{new5.1}) is to find an optimal portfolio $\pmb{x}$ such that in the worst-case with respect to $\Theta_{\mathcal{S}^N}$, one can get the highest possible lower bound $t$ on the portfolio return $\tilde{\pmb{\xi}}^T\pmb{x}$ with probability at least $1-\epsilon$.

Within our framework (\ref{1.5}) and Algorithm \ref{algo1}, the BDRCC problem (\ref{new5.1}) can be approximated by the following classical RO problem:
\begin{equation}
	\max_{\pmb{x}}\left\{\min_{\pmb{\xi}\in\Xi_\epsilon}\pmb{\xi}^T\pmb{x}:\pmb{e}^T\pmb{x}=1,\pmb{x}\geq0\right\},\label{5.1}
\end{equation}
where the uncertainty set $\Xi_\epsilon=\Xi({\mathcal{S}^{N}},\epsilon,\alpha)$ constructed by the framework in Section 3 will imply the BDRCC in (\ref{new5.1}). That is, the optimal value of problem (\ref{5.1}) will provide a conservative bound of return with probability at least $1-\alpha$  w.r.t. $\mathbb{P}_{\pmb{\theta}^c}$ for the optimal portfolio $\pmb{x}^*$ of problem (\ref{new5.1}). 

To demonstrate the practical value of problem (\ref{5.1}), we first show the superiority of our Bayesian type uncertainty set $\Xi({\mathcal{S}^{N}},\epsilon,\alpha)$ by comparing it with the uncertainty sets derived in  \cite{main,CS}.
\subsubsection{Portfolio selection with simulated data}
In order to make comparisions on the same basis, we consider $d = 20$ assets with the same setting as that in \cite{main,CS}. Specifically, we assume that the return of each asset follows the following binomial distribution:
\begin{equation}
	 \tilde{\xi}_i=\left\{
\begin{split}
& \frac{\sqrt{(1-\theta_i)\theta_i}}{\theta_i}\ \ &\text{with probability }\theta_i,\\
& -\frac{\sqrt{(1-\theta_i)\theta_i}} {1-\theta_i}& \text{with probability }1-\theta_i.
\end{split}
\right.\label{bio}
\end{equation}

It is easy to see that all the assets have the same mean return,0\%, and  the same standard deviation, 1.00\%. However,  the different parameters $\theta_i$, $i=1,\cdots,d$ reflect different degrees of asymmetry of individual return distributions.  Here the $\theta_i$ is chosen as
\[\theta_i=\frac{1}{2}(1+\frac{i}{d+1})\ \ i=1,\cdots,d.\]
Therefore, a higher index $i$ results in a larger loss with a smaller probability of the corresponding risky asset.

As an illustration, we consider two kinds of uncertainty sets $\Xi^I({\mathcal{S}^{N}},\epsilon,\alpha)$ and $\Xi^M({\mathcal{S}^{N}},\epsilon,\alpha)$, denoted as $\Xi^I$ and $\Xi^M$ for brevity,  derived in Theorem \ref{thm5.1} and Theorem \ref{thm5.4} respectively.  To compare with the existing methods, we will also consider the uncertainty set motivated by forward and backward deviations, denoted as $\Xi^{I,BG}$, in \cite{main} and the uncertainty sets $\Xi^{M,BG}$  and $\Xi^{M,CS}$ built from marginal samples in \cite{main} and  \cite{CS}, respectively. These methods all aim at constructing an uncertainty set which can imply the DRCC at level $1-\epsilon$ with a high probability $1-\alpha$.  The degree of credibility $\alpha$ and the level for probability guarantee $\epsilon$ are set as $\alpha=\epsilon=10\%$ in this experiment. 

We first compare $\Xi^I$ and $\Xi^{I,BG}$. Figure 1 depicts the sets $\Xi^I$ (blue ones) and $\Xi^{I,BG}$ (red ones) constructed under different numbers of samples, $N$. For reference, we also plot the set based on true distribution  $\mathbb{P}_{{\pmb{\theta}}^c}$ (black dot) in this example. To visualize our results, we here select two components $\tilde{\xi}_1$ and $\tilde{\xi}_{20}$ for presentation.
\begin{figure}[H]  
	\center{\includegraphics[width=10cm]  {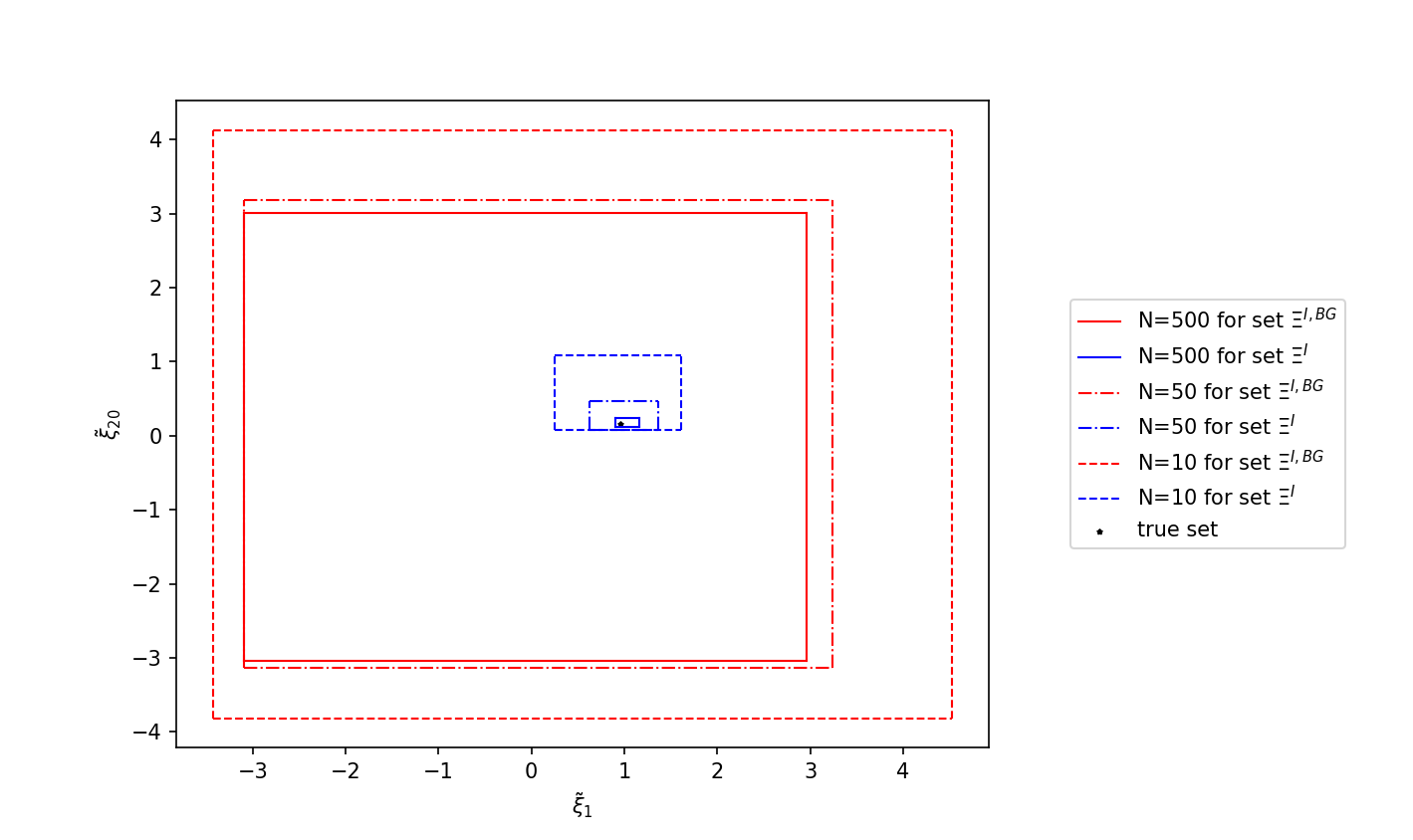}}   
	\renewcommand{\figurename}{Figure}
	\caption{\label{1} The uncertainty sets $\Xi^I$ and $\Xi^{I,BG}$ constructed with different sample sizes}   
\end{figure}
We also compare $\Xi^M$ with $\Xi^{M,BG}$ and $\Xi^{M,CS}$. This situation is considered to reflect the case without any prior knowledge about the dependence of marginal distributions. Similarly, we construct the sets  $\Xi^M$ (blue ones), $\Xi^{M,BG}$ (red ones) and  $\Xi^{M,CS}$ (green ones) for different $N$s, the number of samples. We also plot the set based on true distribution  $\mathbb{P}_{{\pmb{\theta}}^c}$ (black dot) in this example. The results are displayed in Figure 2, again with respect to two components $\tilde{\xi}_1$ and $\tilde{\xi}_{20}$ for presentation.

\begin{figure}[H]  
	\center{\includegraphics[width=14cm]  {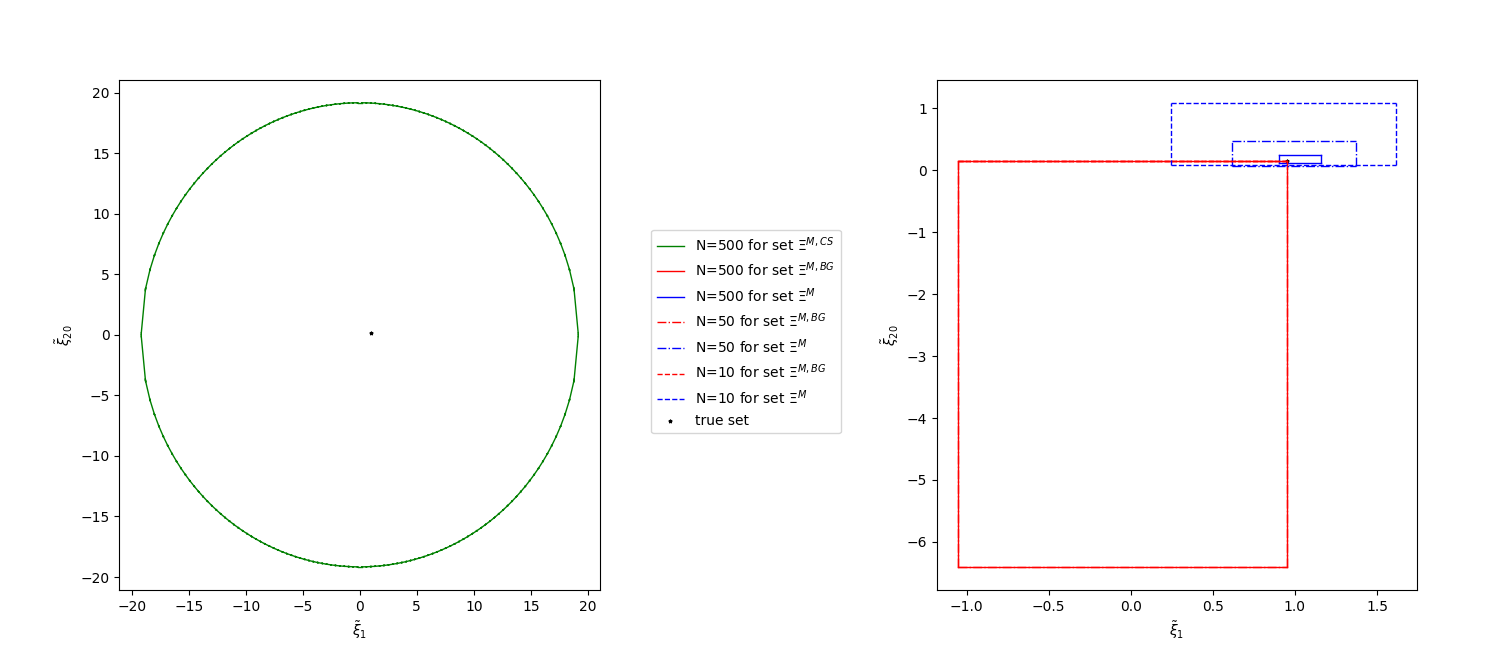}}   
	\renewcommand{\figurename}{Figure}
	\caption{\label{1} The uncertainty sets $\Xi^M$, $\Xi^{M,BG}$ and $\Xi^{M,CS}$ constructed with different sample sizes}   
\end{figure}

Figures 1 and 2 show that the uncertainty set constructed with our schema would converge to the true uncertainty set with the increase of the sample size. Nevertheless, we cannot get the set $\Xi^{M,CS}$ when $N=10$ and $N=50$ for the reason that the approach in \cite{CS} is only applicable for the large data set, otherwise, the volume of the constructed uncertainty set $\Xi^{M,CS}$ will become infinity. Meanwhile, the uncertainty set $\Xi^{M,BG}$ constructed for $\tilde{\xi}_1$ and $\tilde{\xi}_{20}$ when $N=10$ is 
\[\{(\xi_{1},\xi_{20}):-1.0488\leq\xi_{1}\leq0.9534,\ \xi_{20}=  0.1561\},\]
whereas the corresponding uncertainty set $\Xi^{M,BG}$  when $N=50$ and $N=500$ is
\[\{(\xi_{1},\xi_{20}):-1.0488\leq\xi_{1}\leq0.9534,\ -6.4031\leq\xi_{20}\leq 0.1561\}.\] 
This is intuitively unacceptable. As the sample size increases, our knowledge or information about the true distribution should be enriched. Therefore, the uncertainty set should become smaller with a larger sample size.  However, the above uncertainty set $\Xi^{M,BG}$ for $N=50$ or $N=500$ is larger than that for $N=10$. These phenomena reflect the fact that most existing data-driven methods for constructing uncertainty sets are built on the premise of a large enough data set. The reason for the above unreasonable results is that the approaches in \cite{main} and \cite{CS} cannot construct a reasonable ambiguity set of distributions which can contain the underlying true distribution $\mathbb{P}_{\pmb{\theta}^c}$ with high probability in a proper way.

Further dissecting the results in Figure 1 and Figure 2, we have the following observations: 

$\bullet$ Our Bayesian-type method can construct significantly smaller uncertainty sets than those of the traditional DRCC methods, under the same finite-sample probability guarantee. This is because existing researches lack a good significance test for general distributions, resulting in overly conservative uncertainty sets.

$\bullet$ It can be seen from Figures 1 and 2 that for existing approaches, different methods are required to construct uncertainty sets under different circumstances. However, from a Bayesian perspective, we convert distribution uncertainty to parameter uncertainty, and construct uncertainty sets for parameters, which provides a generic framework for constructing uncertainty sets under different distribution assumptions. Additionally, for our method, the uncertainty set obtained under independent marginal distributions (Figure 1) is a subset of the set obtained under the dependence case (Figure 2), which cannot be realized by existing approaches.

$\bullet$ Our method is applicable for both large and small sample cases and shows excellent performance constantly. However, most existing data-driven methods are either unsuitable for small samples or exhibit unreasonable performance.

Furthermore, with the constructed convex uncertainty sets, it is easy to solve the RO optimization problem (\ref{5.1}). We generate two in-sample data sets with the distribution (\ref{bio}) for sizes $N=500$ and $N=2000$, respectively. To test the performance of the optimal portfolios obtained by different methods with the in-sample data set, we generate the out-of-sample data set with the distribution (\ref{bio}) for size $N=50$. To ensure that our numerical results are not specific to a particular set of samples, we repeat the solution procedure under each sample size 100 times. The results are shown in Table 1 and Table 2, where I(I,BG) stands for our method (the method in \cite{main}) for the independent marginal distribution case, M(M,CS and M,BG) stands for our method (the method in \cite{CS} and \cite{main}, respectively) for the dependent marginal distribution case. Column $v_{in}$ shows the average in-sample optimal value of problem (\ref{5.1}), while column $v_{out}$ shows the average out-of-sample 10\% worst-case return of the obtained optimal portfolio. As $\Xi^{I}$ is a subset of $\Xi^{M}$, the in-sample 10\% worst-case return of I in Table 1 is better than that of M in Table 2. 

Based on the results presented in Tables 1 and 2, it is clear that $v_{out}$ under both I and M outperform that of the comparing methods for the out-of-sample 10\% worst-case return. We note that the performance of $\Xi^{M,BG}$ depends slightly on $N$. Except for the M,BG case, the performance of other methods becomes better in terms of both $v_{in}$ and $v_{out}$ when the sample size $N$ increases. As mentioned in \cite{main}, the in-sample optimal value $v_{in}$ provides a loose bound, we prefer to select the best uncertainty set with respect to the out-of-sample performance $v_{out}$.  Therefore, given that the uncertainty set constructed by our method consistently leads to a larger $v_{out}$, we assert that our method can generate a better, less conservative uncertainty set, and thus can find a better, robust optimal portfolio.
\begin{table}[]
			\centering
	 
	\caption{Portfolio performance under different methods for independent returns}
	\setlength{\tabcolsep}{8mm}{
	\begin{tabular}{@{}lllll@{}}
		\toprule
		& \multicolumn{2}{c}{$N=500$}                          & \multicolumn{2}{c}{$N=2000$}                          \\ \cmidrule(l){2-5} 
		& \multicolumn{1}{c}{$v_{in}$} & \multicolumn{1}{c}{$v_{out}$} & \multicolumn{1}{c}{$v_{in}$} & \multicolumn{1}{c}{$v_{out}$} \\ \midrule
		I,BG & -1.1925                & -1.1273                 & -1.3734                 & -1.0818                 \\
		I    & -1.0622                & -1.0261                 & -1.0529                 & -0.9803                 \\ \bottomrule
	\end{tabular}}
\end{table}
\begin{table}[]
			\centering
	 
	\caption{Portfolio  performance under different methods for dependent returns}
	\setlength{\tabcolsep}{8mm}{
	\begin{tabular}{@{}lllll@{}}
		\toprule
		& \multicolumn{2}{c}{$N=500$}                          & \multicolumn{2}{c}{$N=2000$}                          \\ \cmidrule(l){2-5} 
		& \multicolumn{1}{c}{$v_{in}$} & \multicolumn{1}{c}{$v_{out}$} & \multicolumn{1}{c}{$v_{in}$} & \multicolumn{1}{c}{$v_{out}$} \\ \midrule
		M,BG & -1.0488                & -1.0488                 & -1.0488                 & -1.0488                 \\
		M,CS & -1.2293                & -1.0813                 & -1.1704                 & -1.0085                 \\
		M    & -1.0633                & -1.0090                 & -1.0539                 & -0.9797                 \\ \bottomrule
	\end{tabular}}
\end{table}

\subsubsection{Portfolio selection with real trading data}
From a technical perspective, the previous example with simulated data demonstrated that our approach is superior to typical approaches when it comes to approximating DRCC. However, real-world situations can be more complex and unpredictable. To further showcase the effectiveness of our method, we now consider a more sophisticated portfolio selection problem with real trading data.

We consider a portfolio selection problem with 14 stocks from different sectors in the USA stock market, including technology, financials, and pharmaceuticals. The concrete stock pool includes Visa, JPM, KO, JNJ, PG, INTC, AMZN, META, AAPL, MSFT, BA, CMCSA, CVX, and PFE, whose random returns are denoted by $\tilde{\xi}_1,\tilde{\xi}_2,\cdots,\tilde{\xi}_{14}$.  To better model the multivariate dependency among the returns of these stocks, we adopt the Gaussian copula, similar to that in papers like \cite{copulas},  which is a widely used tool in practice and also provides a standard benchmark for evaluating portfolio performance.

	Considering the limited research on portfolio selection with copulas under the DRCC framework, we compare our approach (M) with two other methods: the sample average approximation approach to chance constrained portfolio selection problem (CCP) in \cite{chance} and the robust portfolio optimization (RPO) with copulas in \cite{copulas}.  One can refer to \cite{chance} and \cite{copulas} for the detailed algorithms used to solve the resulting CCP and RPO problems, respectively, due to space limitations.

In the following experiment, we set the probability guarantee level at $\epsilon=5\%$ and the degree of credibility $\alpha=10\%$.
We estimate the joint distribution using historical daily returns. To demonstrate the robustness of our approach and the effect of different amounts of historical return data, we adopt three groups of data to estimate the joint distribution, respectively: for a period of one year from March 25, 2021, to March 22, 2022, denoted $N=250$, a period of two years from March 30, 2020, to March 22, 2022, denoted as $N=500$ and a period of four years from April 1, 2018, to March 22, 2022, denoted as $N=1000$. To evaluate the effectiveness of our portfolio selection method, we select the S\&P 500 index as the benchmark and show the cumulative out-of-sample returns from March 23, 2022, to January 10, 2023, of the optimal portfolios got with different methods in Figures 3-5.
Moreover, to comprehensively compare the performances of optimal portfolios determined with different approaches, we show in Table 3 typical performance indexes including the in-sample optimal value and the out-of-sample optimal return in columns $r_{in}$ and $r_{out}$, the maximum drawdown rate in column MaxDrawdown and the Sharpe ratio in column Sharpe.
\begin{table}[]
		\centering
		\caption{Performance of optimal portfolios under different methods and sample sizes}
		\setlength{\tabcolsep}{4mm}{
\begin{tabular}{@{}cccccc@{}}
\toprule
                  sample size      & method & $r_{in}$    & $r_{out}$    & MaxDrawdown &  Sharpe \\ \midrule
\multirow{3}{*}{250}  	&	CCP   & 0.0019                              & -0.2155 & 0.2831         & 0.0520        \\
	&	RPO  & 0.0004                              & -0.1292 & 0.2289         & 0.0602        \\
	&	M    & -0.0100 & 0.0035  & 0.1386         & 0.0724        \\\midrule
\multirow{3}{*}{500}  	&	CCP   & 0.0013                             & -0.1892 & 0.2489         & 0.0547        \\
	& RPO  & 0.0007                             & -0.0744 & 0.2242         & 0.0654        \\
	&	M    & -0.0094& 0.0009  & 0.2441         & 0.0702        \\\midrule
\multirow{3}{*}{1000} &  CCP   & 0.0015    & -0.0174 & 0.1981         & 0.0695        \\
		& RPO   & -0.0003   & -0.0635 & 0.1865         & 0.0663        \\
		& M    & -0.0067 & 0.0302  & 0.1784         & 0.0725    \\\bottomrule
    benchmark    & S\&P500 &     -     & -0.1368 & 0.2270         & 0.0592         \\ \bottomrule
\end{tabular}}
\end{table}

\begin{figure}[h]  
	\center{\includegraphics[width=12cm]  {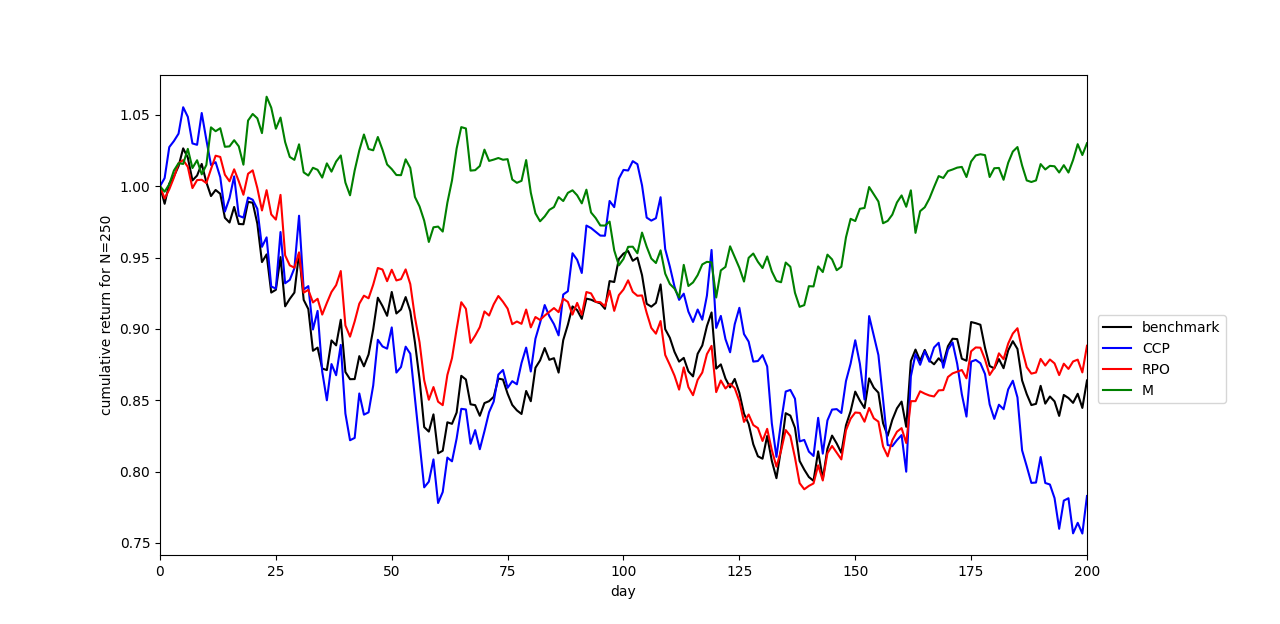}}   
	\renewcommand{\figurename}{Figure}
	\caption{\label{1}  Out-of-sample cumulative returns of optimal portfolios with different methods, $N=250$}   
\end{figure}
\begin{figure}[h]  
	\center{\includegraphics[width=12cm]  {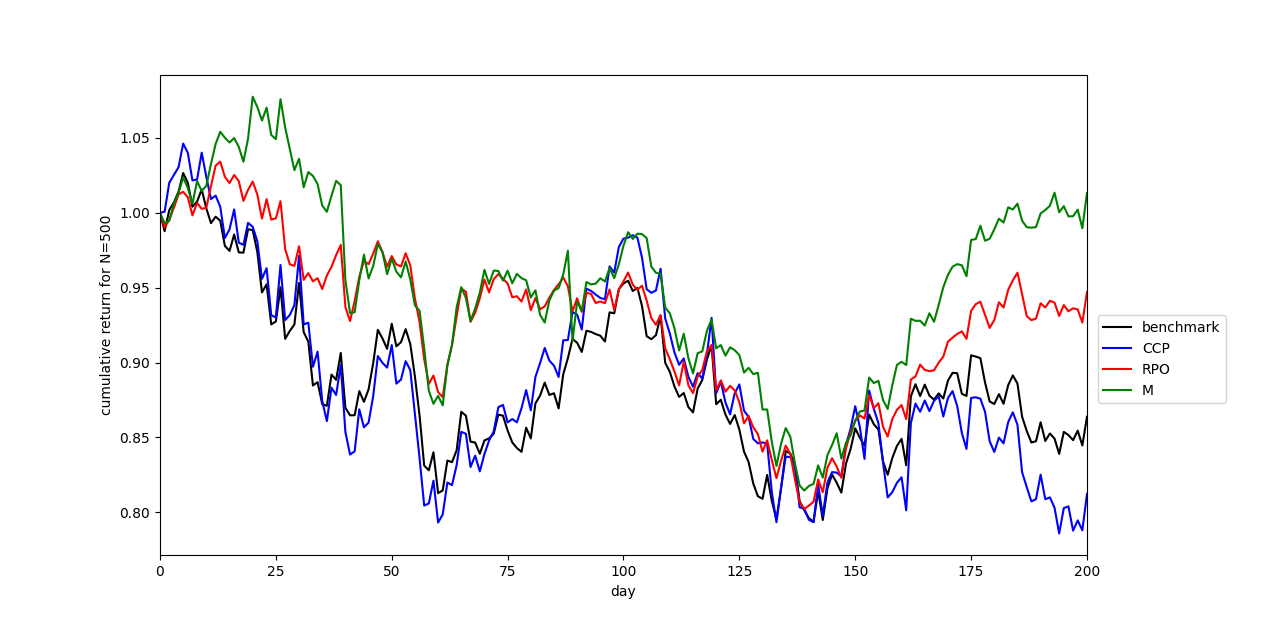}}   
	\renewcommand{\figurename}{Figure}
	\caption{\label{1}   Out-of-sample cumulative returns of optimal portfolios with different methods, $N=500$}   
\end{figure}
\begin{figure}[h]  
	\center{\includegraphics[width=12cm]  {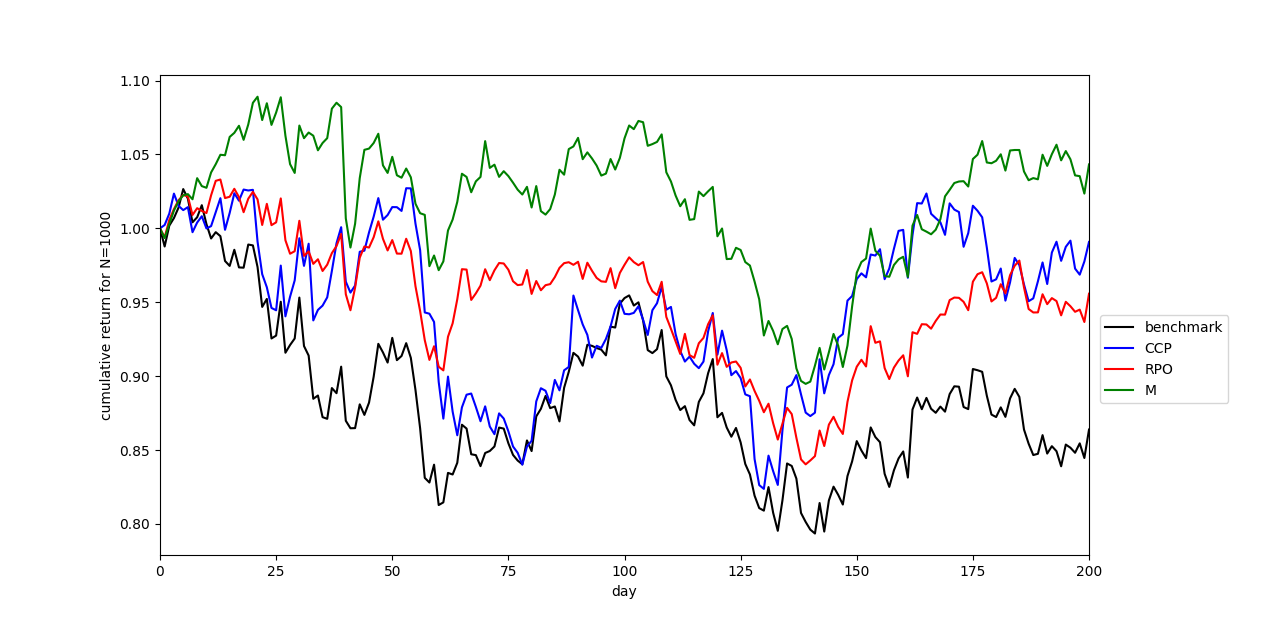}}   
	\renewcommand{\figurename}{Figure}
	\caption{\label{1}  Out-of-sample cumulative returns of optimal portfolios with different methods, $N=1000$}   
\end{figure}

Since each approach has its specific objective function, we do not examine the difference in optimal values $r_{in}$. First, it is obvious from Table 3 and Figures 3-5 that our method outperforms CCP, RPO and the benchmark S\&P 500 in terms of the out-of-sample return $r_{out}$ for all different historical data sizes. Second, as shown in Table 3, the optimal portfolio got with our method has a significantly higher Sharpe ratio than other methods and the benchmark, and, in most cases, our method can better control the downside risk in terms of the maximum drawdown rate.

Our approach, through an estimator of the BDRCC, can focus not only on its out-of-sample portfolio return but also on its control for tail risk. That is, it is important to note that the optimal value $r_{in}$ got with our method represents a lower bound of the optimal return with a probability at least $1-\epsilon$. We refer to the $\epsilon$ percentile for the optimal portfolio return as $r_\epsilon^*$. As an additional metric to show the reasonability and practical value of our method, we examine the difference between $r_{in}$ and $r_\epsilon^*$. To evaluate the deviation between the estimated lower bound $r_{in}$ and the true $\epsilon$ percentile $r_\epsilon^*$, we consider two indicators named deviation and relative deviation, respectively: $d=r_\epsilon^*-r_{in}$ and $D=\frac{r_\epsilon^*-r_{in}}{r_\epsilon^*}$. Obviously, the larger the value of $d$ $(D)$, the greater the error of the bound estimate would be.

Table 4 and Figure 6 show the deviation results of our BDRCC approach with respect to different degrees of credibility $\alpha$ and varying data sizes $N$.
\begin{table}[]
	\centering	 
	\caption{Deviations under different credible levels and data sizes}
	\setlength{\tabcolsep}{3mm}{
	\begin{tabular}{@{}ccccccc@{}}
		\toprule
		& \multicolumn{2}{c}{$N=250$}                & \multicolumn{2}{c}{$N=500$} & \multicolumn{2}{c}{$N=1000$} \\ \cmidrule(l){2-7} 
		& $d$                              & $D$    & $d$          & $D$      &$d$          & $D$      \\ \midrule
		M, $\alpha=100\%$ & 0.0089                         & 17.5966 & 0.0088      & 16.3362     & 0.0058      & 9.2738       \\
		M, $\alpha=50\%$  & 0.0101                         & 19.8641 & 0.0096      & 17.8244     & 0.0069      & 10.9560      \\
		M, $\alpha=10\%$  &0.0105 & 20.7648 & 0.0099      & 18.4155     & 0.0073      & 11.6242      \\
		M, $\alpha=5\%$   & 0.0107                         & 21.0620 & 0.0100      & 18.6105     & 0.0075      & 11.8447      \\
		M, $\alpha=1\%$   & 0.0110                         & 21.6680 & 0.0102      & 19.0083     & 0.0077      & 12.2943      \\ \bottomrule
	\end{tabular}}
\end{table}
\begin{figure}[h]  
	\center{\includegraphics[width=15cm]  {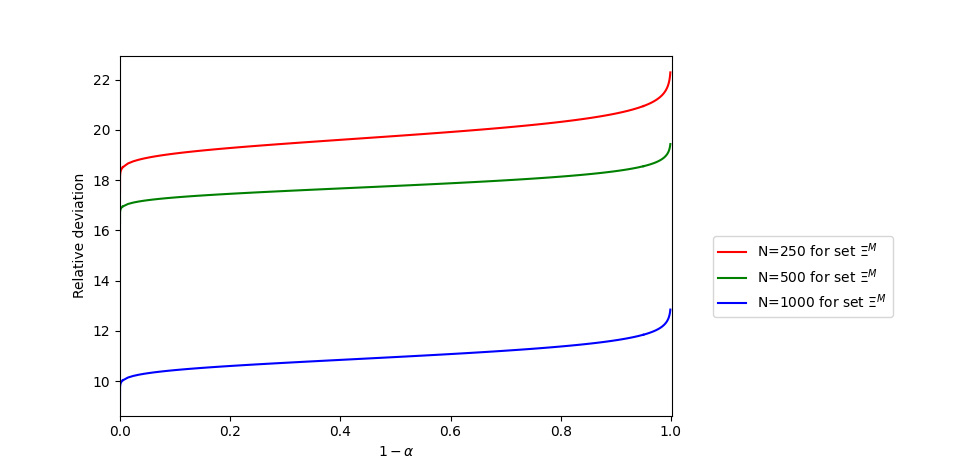}}   
	\renewcommand{\figurename}{Figure}
	\caption{\label{1}  Relative deviations  with different credible levels and data sizes}   
\end{figure}
The positive values in Table 4 mean that our method produces a return, at different credible levels $\alpha$, that is lower than the true percentile, which indicates that our approximation is always more conservative. Considering that different credible levels lead to uncertainty sets with different sizes, we compare the relative deviation of different bounds under different $\alpha$s and  $N$s in Figure 6. It can be seen from Figure 6 that with the decrease of $\alpha$, the ambiguity set of distributions becomes larger, and so does the resulting relative deviation $D$. Meanwhile, the relative deviation monotonically decreases with the increase of the historical data size.

It is also noteworthy that increasing the historical data size $N$ significantly enhances the performance of our method and the relative deviation $D$. The reason behind this is that, more samples enable us to obtain more information about the return distributions and achieve a better approximation of the return distributions, resulting in a better solution for our BDRCC problem. This again demonstrates the reasonability and practical value of our approach.
\subsection{Queue}
As another application of our schema, we study the behavior of the generated approximate BDRCC problem for the robust queueing model examined in \cite{ref4,ref14}. Actually, we will combine the methods therein with our new uncertainty set construction method to generate an upper bound on the waiting time of a queueing network. Specifically, we focus on the waiting time, as the objective function, in an M/M/1 queue.

Suppose  $\tilde{\xi}_i$ is a 2-dimensional  random variable whose components $(\tilde{x}_i,\tilde{t}_i)$ for $i = 1,\cdots,n$ stand for the uncertain service times and inter arrival times of the first $n$ customers in a queue. We assume that $\tilde{\xi}_i$ for $1\leq i\leq n$ are independent and the components $\tilde{x}_i$ and $\tilde{t}_i$ are also independent. Then the random waiting time of the $n$th customer can be recursively defined as (\cite{ref36})
\begin{equation}
	\tilde{W}_n=\max_{1\leq j\leq n}\left(\max\left(\sum_{l=j}^{n-1}\tilde{x}_l-\sum_{l=j+1}^{n}\tilde{t}_l,0\right)\right)=\max\left(0,\max_{1\leq j\leq n}\left(\sum_{l=j}^{n-1}\tilde{x}_l-\sum_{l=j+1}^{n}\tilde{t}_l\right)\right).\label{que}
\end{equation}

Bertsimas et al. \cite{ref4} consider the following worst-case for the waiting time
\begin{equation}
	\max\left(0,\max_{1\leq j\leq n}\max_{(\pmb{x},\pmb{t})\in\Xi}\left(\sum_{l=j}^{n-1}\tilde{x}_l-\sum_{l=j+1}^{n}\tilde{t}_l\right)\right).\label{5.2}
\end{equation}
Inspired by their study, we would like to demonstrate the excellence of our method by comparing the optimal values of problem (\ref{5.2}) with the uncertainty set $\Xi$ constructed by different methods.

Similar to many studies like \cite{queue2,queue1}, we assume that the uncertain service time follows an exponential distribution with parameter $\theta_1$ and the inter arrival time follows a Poisson distribution with parameter $\theta_2$. Thus the credible intervals for $\theta_1$ and $\theta_2$ can be easily computed as $\Theta_{\alpha',1}=[\theta_{l,1},\theta_{r,1}]$ and $\Theta_{\alpha',2}=[\theta_{l,2},\theta_{r,2}]$ where $\alpha'=1-\sqrt{1-\alpha}$,  $\theta_{l,i}=I(\hat{\theta}_i)^{-1/2}z_{1-\frac{\alpha'}{2}}-\hat{\theta}_i$ and 
$\theta_{r,i}=I(\hat{\theta}_i)^{-1/2}z_{1-\frac{\alpha'}{2}}+\hat{\theta}_i$, $i=1,2$. Applying our uncertainty set $\Xi^{I}({\mathcal{S}^{N}},{\bar{\epsilon}/n},\alpha)$ and Theorem \ref{thm5.1} to the inner maximization in (\ref{5.2}) leads to
\begin{equation}
	\max_{1\leq j\leq n}(\text{VaR}_{1-\sqrt[d]{1-\bar{\epsilon}/n}}^{\mathbb{P}_{{\theta}_{r,1}}}(\tilde{x})-\text{VaR}_{1-\sqrt[d]{1-\bar{\epsilon}/(n-1)}}^{\mathbb{P}_{{\theta}_{l,2}}}(\tilde{t}))(n-j),
\end{equation}
which has a closed form solution
\begin{equation}
W_n^{1,I}=(n-1)(\text{VaR}_{1-\sqrt[d]{1-\bar{\epsilon}/n}}^{\mathbb{P}_{{\theta}_{r,1}}}(\tilde{x})-\text{VaR}_{1-\sqrt[d]{1-\bar{\epsilon}/(n-1)}}^{\mathbb{P}_{{\theta}_{l,2}}}(\tilde{t})).
\end{equation}
Since each inner maximization in (\ref{5.2}) provides an upper bound to the corresponding $\tilde{x}_l-\tilde{t}_l$ at level $1 - \bar{\epsilon}/(n-1)$ for $\mathbb{P}_{\pmb{\theta}^c}$ with probability at least $1-\alpha$ with respect to $\mathbb{P}_{\mathcal{S}^{N}}$. With $n-1$ individual upper bounds, we can  derive the following BDRCC
\[\inf_{\pmb{\theta}\in\Theta_{\alpha',1}\times\Theta_{\alpha',2}}\mathbb{P}_{\pmb{\theta}}(\tilde{W}_n \leq W_n^{1,I}) \geq 1 - \bar{\epsilon}\]
for $\mathbb{P}_{\pmb{\theta}^c}$ with probability at least $1-\alpha$ with respect to $\mathbb{P}_{\mathcal{S}^{N}}$. This inequality also illustrates that our method will provide a conservative bound for (\ref{que}) with probability at least $1-\alpha$  w.r.t. $\mathbb{P}_{\pmb{\theta}^c}$.

For comparison, we also consider the solution of problem (\ref{5.2}) with the uncertainty set proposed in \cite{main}. In their study, the closed-form of the $1- \bar{\epsilon}$ quantile of the waiting time can be written as
\[W_n^{BG}\equiv\frac{(\log\frac{n}{\bar{\epsilon}})(\sigma_{f1}^2+\sigma_{b2}^2)}{2(m_{b2}-m_{f1})},\]
where $m_{b2}$, $m_{f1}$, $\sigma_{b2}$ and $\sigma_{f1}$ are the thresholds of the confidence region respectively. See \cite{main} for the details about these parameters.

Furthermore, we also apply the method in \cite{ref33} to determine the following bound on the $1-\bar{\epsilon}$ quantile of the waiting time:
\[W^{King}\equiv\frac{\hat{\mu}_x(\hat{\sigma}_{t}^2\hat{\mu}_{x}^2+\hat{\sigma}_{x}^2\hat{\mu}_{t}^2)}{2\bar{\epsilon}\hat{\mu}_t^2(\hat{\mu}_t-\hat{\mu}_x)},\]
where $\hat{\mu}_t$, $\hat{\sigma}_t^2$, $\hat{\mu}_x$ and $\hat{\sigma}_x^2$ denote the means and  variances of all arrival time samples and service time samples, respectively.

To characterize the M/M/1 queue concretely, we let the service time follow an exponential distribution with $\theta_{1}=2$ and the interarrival times follow a Poisson distribution with $\theta_{2}=3.05$. To observe the behavior of the median waiting time for $n=10$ customers under different sample sizes and uncertainty sets, we set $\bar{\epsilon}=50\%$. To examine the performance of different methods, we repeat the estimation process 100 times. Figure 7 shows relevant results.

The dotted curves in Figure 7 stand for the average bound values got with three different methods over the 100 runs with respect to different sample sizes, and the bars in Figure 7 denote the errors between the 10\% and 90\% quantiles. Moreover, sample statistics for median waiting time bounds got with different methods under 10000 samples are shown in Table 5. The second and fourth columns refer to lower and upper quantiles over the simulation. The third column gives the mean value and the last column represents the standard deviation (SD) of the results over 100 runs. 
\begin{table}[h]
	\caption{ Summary statistics of
		median	waiting time bounds for N=10000}  
	\setlength{\tabcolsep}{8mm}{
		\begin{tabular}{@{}ccccc@{}}
			\toprule
			& 10\% & mean  & 90\%  & SD   \\ \midrule
			$W^{1,I}_n$  & 5.2802 & 5.3543  & 5.4287  & 0.0689 \\
			$W^{BG}_n$ & 4.2629 & 7.3267  & 11.4554 & 3.8227  \\
			$W^{King}$  & 9.4126 & 10.1303 & 10.9123 & 0.5319 \\ \botrule
	\end{tabular}}
\end{table}

We can see from Figure 7 and Table 5 that all the bound estimations on the median waiting time improve with the increase of the sample size, as expected; the bound of $W^{BG}_n$  decreases significantly as the sample size increases, and its volatility becomes less and less; the bound of $W^{King}$ is much more stable than that of $W^{BG}_n$, and the average bound $W^{King}$ is less than that of $W^{BG}_n$ when $N<1000$ while the opposite is true for $N=10000$; our bound $W^{1,I}_n$ induced with the Bayesian credible interval is much more stable and  significantly better than both  $W^{BG}_n$  and $W^{King}$, whenever the sample size is large or small.
\begin{figure}[H]  
	\center{\includegraphics[width=12cm]  {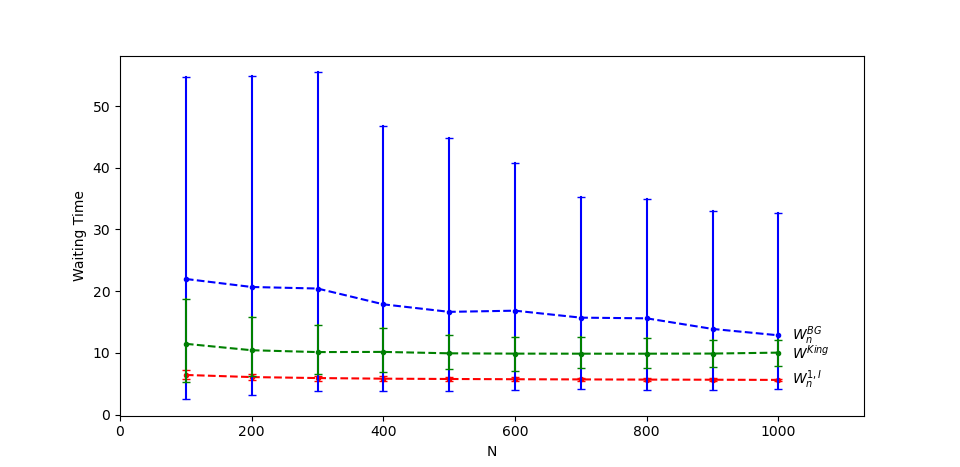}}   
	\renewcommand{\figurename}{Figure}
	\caption{\label{1} Estimations of the median waiting time with different methods and sample sizes}   
\end{figure}
\section{Conclusion}
In this paper, we introduce a new approach for approximating DRCC through constructing uncertainty sets of relevant parameters by using Bayesian credible intervals. Through incorporating the Bayesian credible interval into the uncertainty set, the proposed BDRCC approach can achieve a better trade-off between the data-fitting and robustness of the optimization problem compared to the existent DRCC approaches. Furthermore, our framework provides a unified approach for constructing uncertainty sets under different kinds of dependency among marginal distributions. We prove that, when the sample size tends to infinity, the uncertainty sets based on our BDRCC setting will converge to the true uncertainty set. This gives assurance that our approach provides an optimal approximation of the underlying DRCC problem in the presence of large data. We also show through numerical experiments that our BDRCC approach outperforms the typical methods, even with a small sample size, and can provide more practical, high-quality solutions.

An important research area at present is the multi-stage DRCC problem. Then how to reformulate this multi-stage problem within our framework as a kind of multi-stage RO problem is an interesting and practical issue worth further investigation.

\backmatter

\section*{Statements and Declarations}
\begin{itemize}
\item Funding

This research was supported by the National Key R\&D Program of China (2022YFA1004000, 2022YFA1004001) and National Natural Science Foundation of China under Grant Numbers 11991023, 11901449 and 11735011.

\item Availability of data and materials

The data that support the findings of this study are available from the corresponding author upon request.
\end{itemize}

\end{document}